\documentclass{article}
\usepackage{graphicx} 

\usepackage{amsfonts}
\usepackage{amsmath}
\usepackage[utf8]{inputenc}
\usepackage{xy} \xyoption{all}
\usepackage{tikz}
\usetikzlibrary{shapes,arrows}
\usepackage{ wasysym}
\usepackage{amssymb}
\usepackage{mathrsfs}

\usepackage{multicol}

\usepackage[autostyle]{csquotes}

\newtheorem{theorem}{Theorem}

\newtheorem{fact}[theorem]{Fact}

\newtheorem{corollary}[theorem]{Corollary}

\newtheorem{example}[theorem]{Example}

\newtheorem{lemma}[theorem]{Lemma}

\newtheorem{proposition}[theorem]{Proposition}
\newtheorem{remark}[theorem]{Remark}

\newenvironment{proof}[1][Proof]{\noindent\textbf{#1.} }{\ \rule{0.5em}{0.5em}}

\newcommand{\sslash}{\mathbin{/\mkern-6mu/}}

\def\title#1{{\Large\bf  \begin{center} #1 \vspace{0pt} \end{center}  } }
\def\authors#1{{\large\bf \begin{center} #1 \vspace{0pt} \end{center} } }
\def\university#1{{\sl \begin{center} #1 \vspace{0pt} \end{center} } }
\def\inst#1{\unskip $^{#1}$}

\begin{document}


\title{A Noncommutative Nullstellensatz for  Leavitt Path Algebras
}
\authors{Ayten Ko\c{c}\inst{1} 
\quad 
		Murad \"{O}zayd\i n\inst{2}  
	}
	\smallskip
	
	%
	%

			\university{\inst{1}Gebze Technical  University, Türkiye         }
	\university{\inst{2}University of Oklahoma, USA}

\date{}

\newcommand{\q}[1]{``#1''}

\medskip

\begin{abstract}
\noindent
Hilbert's Nullstellensatz states that a quotient of the algebra of polynomial functions on an algebraic variety over an algebraically closed field is the algebra of polynomial functions on a (sub)variety if and only if its radical is trivial. We prove a noncommutative analog: A quotient of a Leavitt path algebra over an algebraically closed field is isomorphic to a Leavitt path algebra if and only if its radical is trivial. This suggests that a Leavitt path algebra behaves like a noncommutative algebra of polynomial functions on a directed graph. Along the way we provide a complete answer to the question \q{When is a quotient  of a Leavitt path algebra  isomorphic to a Leavitt path algebra?}. We also define a Morita invariant stratification and a parametrization of the ideal space of a Leavitt path algebra and show that a generic quotient of a Leavitt path algebra  is a Leavitt path algebra. We end this article by pointing out a connection with quantum spaces strengthening the analogy with function algebras and several examples involving algebraic quantum spaces illustrating our results.\\

\end{abstract}

\medskip

\medskip

\section{Introduction}

Leavitt path algebras (LPAs) were defined in 2004 in two preprints (published versions \cite{amp07}, \cite{aa05}) as algebraic analogs of graph $C^*$-algebras of functional analysis and named in honor of Leavitt's pioneering works (in the early sixties, culminating in \cite{lea65}) on quantifying the failure of the Invariant Basis Number (IBN) property. In the seventies Cuntz algebras $\mathcal{O}_n$ were defined \cite{Cuntz} as the first concrete examples of separable infinite simple C*-algebras. After several generalizations (such as Cuntz-Krieger algebras) and a connection with symbolic dynamics the di(rected )graph  $\Gamma$ prescribing a generating family and the relations they satisfy became visible.\\

The underlying digraph  $\Gamma$ for both the Leavitt algebra $L(1,n)$ and the Cuntz $C^*$-algebra $\mathcal{O}_n$
consists of one vertex and $n$ loops (the rose with $n$ petals). The generating family (and the relations) of $L(1,n)$ can be identified with the down-sample and up-sample operators of signal processing and the completion of $L(1,n)$ with complex coefficients is $\mathcal{O}_n$. In the last two decades there has been considerable research activity on LPAs and now this area has its own MSCN  (Mathematics Subject Classification Number): 16S88. In addition to the algebras $L(1, n)$ of Leavitt, LPAs include the direct sums of matrix
algebras over fields or Laurent polynomial algebras, algebraic quantum disks
and spheres, and many others. 
The definition of the Leavitt path algebra $L_{\mathbb{F}}(\Gamma)$ of a digraph $\Gamma$ with coefficients in the field $\mathbb{F}$ is given in \ref{LPA} below.  \\

 When  $I$ is a graded ideal of a Leavitt path algebra $L_{\mathbb{F}}(\Gamma)$ it is known that the quotient algebra $L_{\mathbb{F}}(\Gamma)/I$ is isomorphic to a Leavitt path algebra \cite[Theorem 5.7.]{t07}. The original purpose of this note was to answer the natural question \q{When is the quotient algebra  $L_{\mathbb{F}}(\Gamma)/J$ isomorphic to a Leavitt path algebra for an arbitrary ideal $J$ ?} asked to the second named author by M. Kanuni. We give a complete answer in terms of the interaction of the kernel $J$ with the cycles of the digraph $\Gamma$ in Theorem \ref{ana}  which states that the quotient is isomorphic to a Leavitt path algebra if and only if the canonical polynomial generators of the kernel have distinct linear factors. Such ideals are \q{generic} in a sense explained in subsection \ref{3.1} below.\\

In this note by a \enquote{quotient} we mean the morphism $L_{\mathbb{F}}(\Gamma) \longrightarrow L_{\mathbb{F}}(\Gamma)/J$, not just the (isomorphism type of the) object $L_{\mathbb{F}}(\Gamma)/J$. So classifying all quotients of $L_{\mathbb{F}}(\Gamma)$ isomorphic to a Leavitt path algebra amounts to classifying the kernels of such morphisms. In order to talk about genericity we need a space. We define the space of ideals of a Leavitt path algebra via a stratification of the set of ideals and a parametrization of each stratum. \\

There is a filtration on the set of ideals of $L_{\mathbb{F}}(\Gamma)$,  indexed by the lattice of graded ideals with layers consisting of all ideals containing a given graded ideal. The stratification corresponding to this filtration has strata consisting of the ideals $\{ J\}$ with the graded ideal $I$ being the largest graded ideal contained in $J$. A finer filtration is indexed by pairs $(I, \beta)$ where $\beta$ is a collection of cycles in $\Gamma$ with all exits in $I$. The parametrization of the strata of this finer stratification is explained in subsection \ref{3.1} below. \\

This finer stratification and the corresponding parametrization can be defined categorically. Hence the ideal space of a Leavitt path algebra with the given stratification and parametrization is a Morita invariant as explained in Remark \ref{Morita} and the paragraph above it. Also, the stratification by the graded ideals can be defined without mentioning the (standard) $\mathbb{Z}$-grading on $L_{\mathbb{F}}(\Gamma)$; the lattice of graded ideals is a Morita invariant of (the ungraded algebra) $L_{\mathbb{F}}(\Gamma)$.\\

 The main ingredient for proving the necessity of our criterion for a quotient of $L_{\mathbb{F}}(\Gamma)$ to be isomorphic to a Leavitt path algebra is the characterization in Theorem \ref{26} below of finitely generated projective $L_{\mathbb{F}}(\Gamma)$-modules whose endomorphism algebras are finite dimensional. The proof of the sufficiency of our criterion is constructive, we give an explicit digraph $\Gamma  \sslash  J$ and an isomorphism from $L_{\mathbb{F}}(\Gamma) /J$ to  $L_{\mathbb{F}}(\Gamma \sslash J)$.  \\

As a consequence of our criterion for the quotient to be isomorphic to a Leavitt path algebra we show that a generic quotient of a Leavitt path algebra over an algebraically closed field is  isomorphic to a Leavitt path algebra (Theorem \ref{generic}). We also obtain a noncommutative Nullstellensatz for Leavitt path algebras as Corollary \ref{Null}.\\

Of course the Hilbert Nullstellensatz is the original bridge connecting classical algebraic geometry to commutative algebra and the precursor to central development in modern algebraic geometry, arguably also in functional analysis, as we briefly discuss in section \ref{4} of this note. \\
 
 The rest of this note is organized as follows. In   \q{Notation and Preliminaries}   
we give the definitions of  Leavitt path algebras and other relevant concepts. We also state (and sometimes prove) some basic facts that we will use later. Facts we need that are (or should be) well known involving algebras and modules (not specifically  Leavitt path algebras) which are elementary but seem to be difficult to locate in the literature in the form that we will use them, are given in the appendix \cite{ko} which also contains a novel orthogonality relation between projective modules and ideals. \\

In the subsection on ideals we summarize the known characterizations of both graded and arbitrary ideals of Leavitt path algebras in terms of their generators \cite{t07}, \cite{r14}: An arbitrary ideal $J$ of a Leavitt path algebra is generated by 3 types of elements: (i) the vertices in $J$, (ii) idempotents coming from breaking vertices of $J$ and (iii) canonical polynomial generators $f_C(C)$ where $C$ is a cycle in $\Gamma$ all of whose exits are in $J$. Such a cycle $C$ contributes a canonical polynomial generator if and only if $\mathbb{F}[C] \cap J$ is a proper nonzero ideal of $\mathbb{F}[C]$ and  $f_C(x)$ is the minimal polynomial of $C$.\\

In subsection \ref{3.1} we give the aforementioned stratification of the ideal space and the parametrization of each stratum in terms of the canonical polynomial generators of ideals. We show that the lattice of graded ideals is a Morita invariant of the Leavitt path algebra (for row-finite digraphs this was shown in \cite[Proposition 5.2.]{amp07}) by establishing a bijection between the graded ideals of $L_{\mathbb{F}}(\Gamma)$
and the closed submonoids (= order ideals) of $\mathcal{V}(L_{\mathbb{F}}(\Gamma))$, the monoid of isomorphism classes of finitely generated projective  $L_{\mathbb{F}}(\Gamma)$-modules under direct sum. In subsection \ref{3.2} we characterize simple and finitely generated indecomposable projective $L_{\mathbb{F}}(\Gamma)$-modules and describe our parametrization of the ideal space algebraically, without referring to the digraph $\Gamma$. \\

The next subsection containing our main results starts with  the useful fact that  cycles with no exit in a digraph may be replaced with loops without changing the isomorphism type of Leavitt path algebra (Lemma \ref{loop}). After a few lemmas we characterize finitely generated projective (right) $L_{\mathbb{F}}(\Gamma)$-modules whose endomorphism algebras are finite dimensional 
as finite sums of cyclic right ideals generated by \textit{sinks} and \textit{leaks} (defined in subsection \ref{cizge} below). As mentioned above this fact is crucial in the proof of the necessity of our criterion for the quotient of a Leavitt path algebra to be isomorphic to a Leavitt path algebra.\\

Also in subsection \ref{3.2} we establish one-to-one correspondence between isomorphism classes of non-simple finitely generated indecomposable  projective $L(\Gamma)$-modules and the cycles in $\Gamma$ with no exit in Theorem \ref{fgip}. A consequence is the Morita invariance of our stratification of the ideal space explained in Remark \ref{Morita}. \\

In subsection \ref{3.3} we construct a digraph $\Gamma \sslash J$ for every ideal $J$ of the Leavitt path algebra $L_{\mathbb{F}}(\Gamma)$. The quotient $L_{\mathbb{F}}(\Gamma)/J$ is isomorphic to a Leavitt path algebra if and only if every canonical polynomial generator of $J$ is a product of distinct linear factors, in which case $L_{\mathbb{F}}(\Gamma)/J \cong L_{\mathbb{F}}(\Gamma \sslash J)$ by Theorem \ref{ana}. As a consequence, when the coefficient field $\mathbb{F}$ is large enough (for instance, if $\mathbb{F}$ is algebraically closed), the nilradical of $L_{\mathbb{F}}(\Gamma)/J$ equals the Jacobson radical of $L_{\mathbb{F}}(\Gamma)/J$ and this is the obstruction to $L_{\mathbb{F}}(\Gamma)/J$ being a Leavitt path algebra (Theorem \ref{jake} and Corollary \ref{primitif}). This is a noncommutative analog of the classical Hilbert Nullstellensatz in commutative algebra/algebraic geometry stating that a quotient of the algebra of regular functions on an affine variety is the algebra of regular functions on a (sub)variety if and only if its radical is zero. \\

In subsection \ref{3.4} we show that the sufficiency of our criterion  implies that a generic quotient of a Leavitt path algebra is isomorphic to a Leavitt path algebra (Theorem \ref{generic}). When the coefficient field $\mathbb{F}$ is algebraically closed, even if  $L_{\mathbb{F}}(\Gamma)/J$ is not isomorphic to a Leavitt path algebra the ideal $J$ is in the closure of a set of ideals $\{J'\}$
with $L_{\mathbb{F}}(\Gamma)/J' \cong L_{\mathbb{F}}(\Gamma \sslash J)$. Another curious consequence of Theorem \ref{ana} is Corollary \ref{epi}: The kernel of an epimorphism $\varphi: L_{\mathbb{F}}(\Gamma) \longrightarrow L_{\mathbb{F}}(\Gamma')$ is a graded ideal if $\Gamma'$ has neither a sink nor a leak.\\

 Section \ref{4} is mainly a discussion of the noncommutative algebraic geometry of Leavitt path algebras. We compare and contrast relevant results about Leavitt path algebras and graph C*-algebras with some classical theorems (such as the Nullstellensatz, Gelfand-Naimark duality and the Stone-Weierstrass approximation). We also mention a functorial approach to Leavitt path algebras and explain why we consider some digraphs as algebraic quantum spaces. The cofunctor $L_{\mathbb{F}}$ seems to be an \textit{algebraic} quantum analog of the Gelfand-Naimark cofunctor assigning the algebra of complex valued continuous functions to a compact Hausdorff topological space. There is also a quantum analog of the Stone-Weierstrass Theorem in this context \cite[Corollary 7.6]{t07} stating that a completion of a Leavitt path algebra with complex coefficients, with respect to an appropriate norm, is the graph $C^*$-algebra defined by the same digraph. In this analogy Leavitt path algebras are noncommutative counterparts of algebras of polynomial functions. Since some quantum spaces are graph $C^*$-algebras, some Leavitt path algebras can perhaps be thought of as quantum varieties or noncommutative semi-algebraic sets. \\ 
 
 The last section is a collection of examples involving mostly the algebraic quantum spaces mentioned above and illustrating some of our results. While most algebraic properties of Leavitt path algebras are independent of the coefficient field (a partial list is given in the paragraph before Example \ref{Complex}) our criterion for the quotient of a Leavitt path algebra to be isomorphic to a Leavitt path algebra (Theorem \ref{ana}) does depend on the coefficient field, as seen in Examples \ref{Complex}, \ref{ch=2} and \ref{n}. A few more examples illustrate Theorem \ref{ana}. We end this section with some remarks about the correspondence  of the digraphs in our examples with graph $C^*$-algebras and quantum spaces via the cofunctor $L_{\mathbb{F}}$ from the category of digraphs and admissible digraph morphisms (defined in section \ref{4} below) to $*$-algebras and graded $*$-homomorphisms \cite{ht24}. At the level of objects this is the construction of a Leavitt path algebra from a digraph.

\section{Notation and Preliminaries}

We start with the definitions of digraphs, Leavitt path algebras and related concepts, then state and sometimes prove some basic facts that we will need later. 
 \subsection{Digraphs} \label{cizge} 
  
  A {\em di}(\textit{rected} ){\em graph} $\Gamma$ is a four-tuple $(V,E,s,t)$ where $V$ is the set of vertices, $E$ is the set of arrows (directed edges), $s$ and $t:E \longrightarrow V$ are the source and  the target functions. 

\begin{remark}
A digraph is also called an \q{oriented graph} in graph theory, a
\q{diagram} in topology and category theory, a \q{quiver} in representation theory,
usually just a \q{graph} in graph $C^*$-algebras. The notation $\Gamma =(V,E,s,t)$  for a digraph is fairly standard in graph theory, except \q{$A$} for arrows (or arcs) is also  used instead of \q{$E$} for  (directed or oriented) edges. In quiver representations $Q = (Q_0, Q_1, s, t)$ is
more common, however $h$ (denoting head) and $ t$ (denoting tail) and are also used instead of $t$ and $s$, respectively. 
In graph $C^*$-algebras $E = (E^0, E^1, s, r)$ is  used where $r$ stands for \q{range}. We prefer the graph theory
notation which involves two more letters but no subscripts or superscripts. As
in quiver representations we view $\Gamma$ as a small category, so \q{arrow} is preferable
to \q{edge}, similarly for \q{target} versus \q{range} (also range sometimes means image rather than target).
\end{remark}

The digraph $\Gamma$ is \textit{finite} if $E$ and $V$ are both finite. $\Gamma$ is \textit{row-finite} if $s^{-1}(v)$ is finite for all $v$ in $V$. If $s^{-1}(v)= \emptyset$ then  the vertex $v$ is a \textit{sink}; if $t^{-1}(v)= \emptyset$ then $v$ is a \textit{source}. 
If $t(e)=s(e)$ then the arrow $e$ is  a \textit{loop}. 
If $W\subseteq V$ then $\Gamma_{W}$ denotes the \textit{full subgraph} on $W$, that is, $\Gamma_W=(W, s^{-1}(W) \cap t^{-1}(W),  s\vert_{_W},  t\vert_{_W})$.\\

A \textit{path} $p=e_{1}\ldots e_{n}$ of length $n>0$ is a sequence of arrows $e_{1},\ldots ,e_{n}$ such
that $t(e_{i})=s(e_{i+1})$ for $i=1,\ldots ,n-1$. The source of $p$ is $s(p):=s(e_{1})$  and the target of $p$ is $t(p ):=t(e_{n})$.  A path of length 0 is a single vertex $v$ where $s(v):=v $ and $t(v) := v$. We will denote the length of $p$ by $l(p)$ and the set of paths in $\Gamma$ by $Path(\Gamma)$.  An 
\textit{infinite path} $\alpha = e_1e_2 \cdots e_k \cdots$  is an infinite sequence of arrows $e_1, e_2, \cdots ,e_k ,\cdots $
 with $t(e_i)=s(e_{i+1})$ for each positive integer $i$; now $s(\alpha)=s(e_1)$ but $t(\alpha)$ is not defined. An arrow $e$ (respectively, a vertex $v$) is said to be on a path $p$ if $e$ is one of the arrows in the sequence defining $p$ (respectively, if $v$ is $s(e)$ or $t(e)$ for some arrow $e$ on $p$).  An \textit{exit} of a path $p$ is an arrow $e$ which is not on $p$ but $s(e)$ is on $p$.  A path $C=e_1e_2 \cdots e_n$ with $n>0$ is a \textit{cycle} if $s(C )=t(C )$ and $s(e_{i})\neq s(e_{j})$ for $1 \le i\neq j\leq n$. We consider the cycles $e_1e_2 \cdots e_n$ and $e_2e_3 \cdots e_ne_1$ (geometrically) equivalent. 
 The digraph $\Gamma$ is \textit{acyclic} if it has no cycles. \\

The preorder {\bf leadsto}, denoted by $\leadsto$, on $V$ is defined as $u  \leadsto w$  if there is a path $p \in Path(\Gamma)$ from $u$ to $w$, i.e., $s(p)=u$ and $t(p)=w$. The set of \textit{successors} of a vertex $u$ is  
                $$V_{u \> \leadsto}=\{w \in V \> | \>  u \leadsto w \}.$$
 If $X \subseteq V$ then we define $V_{X \leadsto} = \bigcup_{u \in X} V_{u \> \leadsto}\>$.
 The set of \textit{predecessors} of a vertex $u$ is 
 $$V_{ \leadsto u}=\{w \in V \> | \>  w \leadsto u \}.$$   
 If $X \subseteq V$ then we define $V_{\leadsto X} = \bigcup_{u \in X} V_{ \leadsto u}\>$. The full subgraphs on $V_{u \> \leadsto}$, $V_{X\leadsto}$, $V_{ \leadsto u}$ and $V_{ \leadsto X}$  are denoted by $\Gamma_{u \> \leadsto}$, $\Gamma_{X\leadsto}$, $\Gamma_{ \leadsto u}$ and $\Gamma_{ \leadsto X}$ respectively. (The digraph $\Gamma_{u \>\leadsto}$ of the successors of $u$ has been called the \textit{tree of }$u$ and denoted by $T(u)$ in the literature, even though it may contain cycles.)\\

A vertex $v$ is called a {\bf branch  vertex} if $\vert s^{-1}(v) \vert  \geq 2$. (A branch vertex has also been called a \textit{bifurcation vertex} previously.) A {\bf leak} is a vertex $v$ such that  $\Gamma_{v \> \leadsto} $ has neither a cycle nor a branch vertex. A vertex $v$ is a leak if and only if there is a unique path $p_n$ of length $n$ with $s(p_n)=v$ and $t(p_m)\neq t(p_n)$ for all $m\neq n \in \mathbb{N}$. Two leaks $u$ and $v$ are equivalent if they have a common successor, that is, $V_{u \leadsto} \cap V_{v\leadsto} \neq \emptyset$. \\

\noindent
Let $H \subseteq V$ in a digraph $\Gamma$. \\
(i) $H$ is \textit{hereditary} if $\> u \in H\>$ and $\> u \leadsto w$ implies that $w \in H$.\\
 (ii) $H$ is \textit{saturated} if $\> 0<|s^{-1}(u)|< \infty \>$ and  $\> \{t(e) \> | \> e \in s^{-1}(u) \} \subseteq H \>$ implies that $u \in H$.\\

                 The hereditary saturated \textit{closure} $\overline{S}$                of $S \subseteq V$ is the smallest hereditary and saturated subset of $V$ containing S. Since the intersection of hereditary saturated sets are also hereditary and saturated, the hereditary saturated closure of $S$ is the intersection of all hereditary and saturated subsets of $V$ containing $S$. Hereditary saturated subsets of $V$ form a lattice with the meet $X \wedge Y:=X\cap Y$ and the join $X \vee Y:= \overline{X\cup Y}$.

\subsection{Leavitt Path Algebras} \label{LPA}

The {\bf path algebra } $\mathbb{F}\Gamma$ of a digraph $\Gamma=(V,E, s,t)$ with coefficients in the field $\mathbb{F}$ is the $\mathbb{F}$-algebra generated by $V\sqcup E$ satisfying the following relations:\\

\indent(V)  $uv = \delta_{u,v}u$ for all $u,v \in V \> \> \> \quad   $ (where $\delta_{u,v}$ is the Kronecker delta)\\
\indent($E$)  $s ( e ) e  = e=e t( e)  $ for all $e \in E$ \\

\noindent
Alternatively, $\mathbb{F}\Gamma$ may be defined as the $\mathbb{F}$-vector space of formal linear combinations of paths in $\Gamma$ where the product of $p$ and $q$ in $Path(\Gamma)$ is the path $pq$ if $tp=sq$ and  0 otherwise. \\

Given a digraph $\Gamma$ the \textit{doubled  digraph} of $\Gamma$ is $\tilde{\Gamma} := (V,E \sqcup E^*, s~,t~)$ where $E^* :=\{e^*~|~ e\in E \}$,  the functions $s$ and $t$ are  extended as $s(e^{\ast}):=t(e)$ and $t(e^{\ast }):=s(e) $ for all $e \in E$. Thus the dual arrow $e^*$ has the opposite orientation of $e$. We extend $*$ to an operator defined on all paths of $\tilde{\Gamma}$: Let $v^*:=v$ for all $v$ in $V$, $(e^*)^*:=e $ for all $e$ in $E$ and $ (e_1 \ldots e_n)^*:= e_n^* \ldots e_1^*$ for $e_1, \ldots , e_n $ in $E \sqcup E^* $. In particular $*$ is an involution, that is, $**=id$. \\

The \textbf{ Leavitt path algebra} $L_{\mathbb{F}}(\Gamma)
$ of a digraph $\Gamma$ with coefficients in the field  $\mathbb{F}$ is the quotient of the path algebra $\mathbb{F}\tilde{\Gamma}$ by the following relations: \\

\indent(CK1) $e^*f =  \delta_{e,f} ~t(e)\> \> $ for all $e,f\in E$ \\
\indent(CK2) $v= \sum_{s(e)=v} ee^*\> \> $  for all $v$ with $\> \> 0< \vert s^{-1}(v) \vert < \infty $\\

Leavitt path algebras can also be defined as a universal (Cohn) localization of the path algebra. We have chosen to use the more straightforward definition here. \\

\textit{From now on we will omit the parentheses around the inputs of the source and target functions to reduce notational clutter. Also, $L_{\mathbb{F}}(\Gamma)$ will usually be abbreviated to  $L(\Gamma)$ and $dim^\mathbb{F}$ to $dim$ since the coefficient field $\mathbb{F}$ will remain the same until the last section of this note.}  
\\

As a consequence of (CK1) we get that 
$\{ pq^* \vert p, q \in Path(\Gamma) \text{ and } tp=tq\}$ spans $L(\Gamma)$. Another straightforward consequence of (CK1) is:

\begin{fact} \label{p*q} If $p$ and $q$ are paths in $\Gamma$ then  
    $p^*q = 0$ in $L(\Gamma)$ unless $p$ is an initial segment of $q$ or $q$ is an initial segment of $p$. 
\end{fact}

$L(\Gamma)$ is a $\mathbb{Z}$-graded $*$-algebra where $deg(v)=0$ for $v$ in $V$, $deg(e)=1$ and $deg(e^*)=-1$ for $e$ in $E$. This defines a $\mathbb{Z}$-grading on $L(\Gamma)$ since all relations are homogeneous. The $*$-algebra structure on $L(\Gamma)$ is given by extending the operator $*$ to $L(\Gamma)$ as

$$\Big( \sum_{i=1}^n \lambda_ip_iq_i^* \Big)^*:=  \sum_{i=1}^n \lambda_i^*q_ip_i^* $$

\noindent
where $\lambda \mapsto \lambda^*$ is an automorphism of $\mathbb{F}$ with $(\lambda^*)^* =\lambda$ for all $\lambda \in \mathbb{F}$. The default choice for this involutive automorphism is $id_{\mathbb{F}}$ for an arbitrary field, however when $\mathbb{F}=\mathbb{C}$ it is more natural to use conjugation. Since $(ab)^*=b^*a^*$ for all $a, b$ in $L(\Gamma)$ and 
$deg(a^*)=-deg(a)$ for all homogeneous $a \in L(\Gamma)$, we see that
$*$ is a grade reversing involutive anti-automorphism.
 Hence the categories of left modules and right modules are equivalent for $L(\Gamma)$.\\

If $\Gamma$ has finitely many vertices then $L(\Gamma)$ has $1= \sum_{v \in V} v$. In general, $L(\Gamma)$ (also $\mathbb{F}\Gamma$) has local units given by finite sums of distinct vertices since for all $a,\> b \in L(\Gamma)$ 
the sum $u$ of all distinct $sp$ and $sq$ when $a$ and $b$ are expressed as a linear combination of $\{ pq^*\}$ satisfies $ua=a=au$ and $ub=b=bu$. Rings with local units behave similarly as rings with $1$ (see \cite{ko}).\\
 
 In fact, $L(\Gamma)$ has $1$ if and only if $\Gamma$ has finitely many vertices: if $1=\sum_{i=1}^n \lambda_ip_iq_i^*$ and $v \notin \{sp_i\}_{i=1}^n$ then $v 1=0$. So, the necessity of having finitely many vertices uses the basic fact that all the vertices in $\Gamma$ are nonzero elements in  $L(\Gamma)$ which can be shown by constructing a representation of $L(\Gamma)$ sending each vertex to a nonzero operator \cite[Lemma 1.5]{goo09}. The following proposition shows that many elements of a Leavitt path algebra are nonzero.

\begin{proposition} \label{nonzero}
    (i) If $p, q \in Path(\Gamma)$ with $tp=tq$ then $pq^*$ is nonzero in $L(\Gamma)$.\\
    (ii) If $v\in V$ and finite $Z \subsetneqq  s^{-1}(v)$ then $v_Z := v-\sum_{e \in Z} ee^* $ is nonzero in $L(\Gamma)$.\\
    (iii) $ Path(\Gamma)$ is a linearly independent subset of $L(\Gamma)$. Hence the obvious homomomorphism from $\mathbb{F}\Gamma$ to $L(\Gamma)$ is one-to-one.
\end{proposition}
\begin{proof}
    (i) $p^*(pq^*)q=tp$ by (CK1), $pq^* \neq 0$ since   $tp\neq 0$.\\
    
    (ii) If $e\in s^{-1}(v) \setminus Z$ then $v_Ze=e\neq 0$ by (CK1) and (i). Hence $v_Z \neq 0$.\\

(iii) The homomorphism $\varphi$ from $\mathbb{F}\Gamma$ to $L(\Gamma)$ sending each path to itself is a graded homomorphism, so its kernel is a graded ideal. If $\sum \lambda_ip_i$ is a homogeneous element in $\mathbb{F}\Gamma$ then $\varphi (\sum \lambda_ip_i  ) = \sum \lambda_ip_i $ in $L(\Gamma)$ with $l(p_i)=n$ for all $i$, for some $n$. For any (finite) linear combination $\sum \lambda_i p_i$ with $\lambda_i$ nonzero and $l(p_i)=n$ for all $i$ we have $\lambda_j^{-1}p_j^*\sum \lambda_i p_i = tp_j \neq 0$  by Fact \ref{p*q}. Hence $Ker \varphi =0$, equivalently   
$Path(\Gamma)$ is  linearly independent in $L(\Gamma)$. 
\end{proof}\\
  
Proposition \ref{nonzero}(iii) above is in \cite[Lemma 1.6]{goo09}. Using this fact we identify the path algebra $\mathbb{F}\Gamma$ with the subalgebra of the Leavitt path algebra $L(\Gamma)$ generated by the vertices and arrows in $\Gamma$.\\

Except when $v$ is a leak, the following is well-known \cite[Lemma 2.2.7]{aam17}.

\begin{lemma} \label{2.2}

(i) If $v$ is a sink or a leak in $\Gamma$ then $vL(\Gamma)v\cong \mathbb{F}$ where $v \leftrightarrow 1$. \\
    (ii) If $C$ is a cycle in $\Gamma$ with no exit and $sC=v$ then $vL(\Gamma)v\cong \mathbb{F}[x,x^{-1}]$ where $x\leftrightarrow C$. 

\end{lemma}
\begin{proof}
   (i) $vL(\Gamma)v$ is spanned by $pq^*$ with $p,\> q \in Path(\Gamma)$ and $sp = v = sq$. If $v$ is a sink then $p = q = v = pq^*$.  Hence $vL(\Gamma)v \cong \mathbb{F}$ where $v\leftrightarrow 1$. \\

If $v$ is a leak then $vL(\Gamma)v$ is spanned by $\{pp^* \> \vert \> p \in Path(\Gamma), \> sp=v \}$ by the definition of a leak since there is a unique path $p$ of length $n$ with $sp=v$ for all $n\in \mathbb{N}$. If $p=e_1e_2\cdots e_n$ then  $pp^*= v$ by repeated application of (CK2) since $se_i$ is not a branch vertex, so $e_ie_i^*= se_i$ for $i=1,2, \cdots , n$. 
Hence $vL(\Gamma)v\cong \mathbb{F}$ where $v\leftrightarrow 1$. \\

(ii) Since $C$ has no exit, $ee^*=se$ by (CK2) for any arrow $e$ on $C$. So $CC^*=v$ and also $C^*C=v$ by (CK1). Since $v=1$ in $vL(\Gamma)v$ we get $C^{-1}=C^*$ and The homomorphism $\varphi$ from $\mathbb{F}[x,x^{-1}]$ to $vL(\Gamma)v$ is defined as $\varphi (x)= C$. This homomorphism is onto: $vL(\Gamma)v$ is spanned by $\{pq^* \> | \> sp=v=sq, \> tp=tq\}$. 
If $tp \neq v$ then $tp=ee^*$ by (CK2) where $e$ is the unique arrow on $C$ with $se=tp$. So $pq^*=pe(qe)^*$ and we can repeat as needed to get $pq^*=C^mC^{*n}=C^{m-n}$ for some $m$ and $n$ in $\mathbb{N}$. The homomorphism $\varphi$ is also one-to-one since $\{\varphi(x^n)=C^n\}_{n \geq 1}$ is linearly independent (in the path algebra $\mathbb{F}\Gamma$, hence also) in $vL(\Gamma)v \subseteq L(\Gamma)$.
\end{proof}\\

If $M$ is an $L(\Gamma)$-module then $V_M:=\{v\in V \vert \> Mv\neq 0\}$
is called the {\bf support} of $M$ and the {\bf support digraph} $\Gamma_M$ of $M$ is the full subgraph of $\Gamma$ on $V_M$. The support $V_M$ of an $L(\Gamma)$-module $M$ is an isomorphism invariant since an $L(\Gamma)$-module isomorphism $\varphi: M \longrightarrow N$ induces vector space isomorphisms $\varphi \vert_{Mv} : Mv \longrightarrow Nv$ for each $v \in V$. \\

\begin{fact}\label{iso}
    If $p$ is a path in $\Gamma$ then \\
    (i) $tpL(\Gamma) \stackrel{p\underline{\>\>}}{\> \longrightarrow} spL(\Gamma)$ is a split $L(\Gamma)$-monomorphism.\\
    (ii) $spL(\Gamma) \stackrel{p^*\underline{\>\>}\> }{\> \longrightarrow} tpL(\Gamma)$ is a split $L(\Gamma)$-epimorphism. \\
    (iii) If $0< \vert s^{-1}(v)\vert < \infty$ then $\displaystyle \bigoplus_{se=v} teL(\Gamma) \stackrel{\bigoplus e\underline{\>\>}}{\longrightarrow} vL(\Gamma)$ is an isomorphism.\\
(iv) $tpL(\Gamma) \cong pp^*L(\Gamma) =pL(\Gamma)$.

     \end{fact}
    \begin{proof}
       By (CK1) $p^*p=tp$. Hence $p^*\underline{\>\>} \circ p\underline{\>\>}= tp \underline{\>\>} =id_{tpL(\Gamma)}$, thus (i) and (ii) fellow. The inverse of the  $L(\Gamma)$-isomorphism $\displaystyle\bigoplus_{se=v} e\underline{\>\>} \> $ in (iii) is $\displaystyle\prod_{se=v} e^*\underline{\>\>} \>$ since $v=\displaystyle\sum_{se=v} ee^*$ by (CK2).\\

Since  $pp^*L(\Gamma) \subseteq pL(\Gamma) =pp^*pL(\Gamma) \subseteq pp^*L(\Gamma)$ we have  $pp^*L(\Gamma) = pL(\Gamma)$. The homomorphisms $tpL(\Gamma) \stackrel{p\underline{\>\>}}{\longrightarrow} pL(\Gamma)$ and 
$pL(\Gamma) \stackrel{p^*\underline{\>\>}\>}{\> \longrightarrow}  p^*pL(\Gamma)=tpL(\Gamma)$ are inverses of each other since $p^*p =tp$ and $pp^*p=p$. Therefore $pp^*L(\Gamma)=pL(\Gamma) \cong tpL(\Gamma)$, that is, (iv) holds.
       \end{proof}\\

A deep fact going back to a seminal result in \cite{Bergm} and carried over to the context of Leavitt path algebras  in \cite{amp07}
is Theorem \ref{3.2.11} below. Here $v_Z:=v-\sum_{e\in Z} ee^*$ where the vertex $v$ in $\Gamma$ is an infinite emitter and $Z$ is a finite subset of $s^{-1}(v)$. 

\begin{theorem}
    \cite[Corollary 3.2.11]{aam17} \label{3.2.11}
$P$ is a finitely generated projective $L(\Gamma)$-module if and only if $P$ is isomorphic to a finite direct sum of modules of the type $uL(\Gamma)$ with $u\in V$ and $v_ZL(\Gamma)$ where $v_Z=v-\sum_{e\in Z} ee^*$ for an infinite emitter $v$ and a finite subset $Z$ of $s^{-1}(v)$.
\end{theorem}

\subsection{Ideals of Leavitt Path Algebras}

Ideals of Leavitt path algebras are well-studied in the literature, first for finite digraphs then for row-finite digraphs and finally for arbitrary digraphs. Graded ideals are rather special, dividing by a graded ideal yields again a Leavitt path algebra (up to isomorphism). Since vertices are homogeneous elements in a Leavitt path algebra,  any ideal generated by a collection of vertices is a graded ideal. It's easy to check that the set of vertices contained in an ideal is hereditary and saturated. If the digraph $\Gamma$ is row-finite then every graded ideal $I$ of its Leavitt path algebra $L(\Gamma)$ is generated by the set of vertices in $I$. Moreover $L(\Gamma)/I$ is isomorphic to the Leavitt path algebra of the digraph obtained by deleting the vertices $I\cap V$ from $\Gamma$. (This is a consequence of the more general  Theorem \ref{5.7} below.) These statements are no longer true when the digraph has infinite emitters. We need the concept of breaking vertices. \\

Let $\Gamma$ be a digraph and $H$ be a hereditary subset of $V$. A vertex $v$ is called a 
 \textit{breaking vertex} of $H$ \cite[Definition 2.4.4]{aam17} if $v$ belongs to the set 
$$B_H:=\{ v \in V \> \> \vert \> \>   s^{-1}(v) \textit{ is infinite and } 0< \vert s^{-1}(v) \cap t^{-1}(V\setminus H)\vert <\infty\}\> .$$

 \noindent 
$B_H \subseteq V\setminus H$ since $H$ is hereditary. If $v \in B_H$ then $v^H := v- \sum_{e \in Z} ee^*\> $
where $Z:=\{ e\in s^{-1}(v) \> \vert  \> te \notin H\}$, that is, $v^H=v_Z$ in the notation of Proposition \ref{nonzero}(ii).\\

We note that $v^H$ is homogeneous of degree $0$ in the standard $\mathbb{Z}$-grading on
$L(\Gamma)$. For any subset $ S\subseteq B_H$, we define $S^H \subseteq L(\Gamma)$ as $\{ v^H \> \vert \> v \in S\} $. The pair $(H,S)$ is called  \textit{admissible} if $H$ is a hereditary saturated subset of $V$ and $S \subseteq  B_H$. 
   
   \begin{theorem} \cite[Theorem 5.7.]{t07} \label{5.7}
    Let $\Gamma = (V,E, s, t)$ be a digraph, and let $L (\Gamma)$ be its Leavitt path algebra. For an admissible pair $(H,S) $ let $(H\cup S^H)$ denote the ideal generated by $H\cup  S^H$. Then\\
(i) Every graded ideal $I$ of $L(\Gamma)$ is generated by $H\cup S_I^H$, where $H=I \cap V $, and $S_I:=\{v\in B_H \> \vert \> v^H \in I\}$.\\
(ii) If $(H, S)$ is an admissible pair then  $L (\Gamma)/(H\cup S^H) \cong L (\Gamma / (H, S))$ where $\Gamma / (H, S)$ is the digraph defined by
$$V_{\Gamma / (H,S)}:=(V\setminus H) \cup \{v'\> \vert \>  v \in B_H \setminus S\}$$
$$E_{\Gamma / (H,S)} := \{ e \in E \> \vert \> te \notin H\} \cup \{e' \> \vert  \> e \in E, \>  te\in B_H\setminus S \} $$
and $s$ and $t$ are extended to $E_{\Gamma / (H, S)}$ by setting $se' = se$ and $te' = (te)'$.
 \end{theorem}

In particular, every graded ideal of $L(\Gamma)$ is generated by a set of homogeneous idempotents of grade $0$.\\

When $J$ is an arbitrary ideal of $L(\Gamma)\> $ let $H=J\cap V$ and $S_J=\{v\in B_H \> \vert \> v^H \in J\}$ as above. It follows from Theorem \ref{5.7} that the ideal $(H\cup S_J^H)$ is the largest graded ideal contained in $J$. We will denote the digraph $\Gamma / (H,S_J)$ by $\Gamma / J$. 
The digraph $\Gamma/J$ is obtained from $\Gamma$ by deleting the vertices in $H:=J\cap V$, adding a sink $v'$ for each breaking vertex $v$ in $B_H \setminus S_J$ and adding a new arrow $e'$ with $se'=se$ and $te'=v'$ for each arrow $e$ with $te \in B_H \setminus S_J$ .\\

 If  $I=(H \cup S_J^H)$ then $\Gamma /J= \Gamma/ I$. We have a  canonical epimorphism from  $L(\Gamma / J) \cong L(\Gamma)/ (H\cup S_J^H)$ to
$L(\Gamma)/J$ mapping $v \in B_H \setminus S_J$ to $v-v^H+J$, the corresponding new sink $v'$ to $v^H+J$, each arrow $e$ with $te \in B_H \setminus S_J$ to $e(v-v^H) +J$ and the corresponding arrow $e'$ to $ev^H +J$. \\
       
Arbitrary ideals of Leavitt path algebras of row-finite digraphs were characterized in terms of their generators in  \cite{c11}. This was  extended to arbitrary digraphs in \cite{abcr12}. This result was improved and put in a more concise form (some redundant generators were disposed of) in \cite{r14}. Here is an equivalent version which we will use:

\begin{theorem} \label{Ranga}
If $J$ is an ideal of the Leavitt path algebra $L(\Gamma)$ with $\Gamma$ an arbitrary digraph then $J$ is generated by \\
(i) $H:=J\cap V\> $, $\quad$ (ii) $S_J^H$,\\
(iii) the canonical polynomial generators $\{f_C(C)\}$ where $C$ is a cycle in $\Gamma$ such that $J\cap \mathbb{F}[C]$ is a proper nonzero ideal of $\mathbb{F}[C] \cong \mathbb{F}[x]$ and $f_C(x)\in \mathbb{F}[x]$ is the unique  generator with $f_C(0)=1$ of the ideal in $\mathbb{F}[x]$ corresponding to the ideal $J\cap \mathbb{F}[C]$ in $\mathbb{F}[C]$. 
\end{theorem}

The cycles $C$ corresponding to the canonical polynomial generators have no exit in $\Gamma / J$, hence such $C$ are exclusive in $\Gamma$. If $C=e_1e_2 \cdots e_n$ with $f_C(C)$ a canonical polynomial generator then $e_1^*f_C(C)e_1=f_C(D)$ in $L(\Gamma/J)$ where $D=e_2 \cdots e_ne_1$. Therefore the polynomial $f_C(x) \in \mathbb{F}[x]$ depends only on the geometric cycle $C$. If $C$ is a cycle corresponding to a polynomial generator of $J$ then the corner algebra $sCL(\Gamma/J)sC \cong \mathbb{F}[x,x^{-1}]$ by Lemma \ref{2.2}(ii).  The polynomial $f_C(C) \in \mathbb{F}[C]$ is the unique generator of $J\cap sC L(\Gamma /J)sC$ with  constant term $sC$. In fact, $f_C(x)= det ( id-xC^*)$ where $C^*:= \underline{\> \> \> }C^*$ is the linear transformation on the corner algebra $sC\big(L(\Gamma)/J\big) sC$ given by right multiplication with $C^*$. \\

A generalization of the Cuntz-Krieger Uniqueness Theorem \cite[Theorem 2.2.16]{aam17}  immediately follows (if $\Gamma$ satisfies Condition (L) then (ii) below is vacuously true):
\begin{corollary}
 If $\varphi$ is a homomorphism from  $L(\Gamma)$ to an algebra $A$ then $\varphi$ is one-to-one if and only if the following conditions are satisfied:\\
 (i) $\varphi(v)\neq 0$ for all vertices $v$ in $\Gamma$;\\
 (ii)  for each cycle $C$ with no exit in $\Gamma$ the set $\{\varphi (C^n) \}_{n=1}^\infty$ is linearly independent.
\end{corollary}
\begin{proof} 
If $\varphi$ is one-to-one then (i) and (ii) are satisfied by Proposition \ref{nonzero}. Conversely, if (i) is satisfied then $H:=V \cap Ker\varphi =\emptyset$ hence there are no breaking vertices of $H$. If (ii) is also satisfied then $Ker \varphi$ has no canonical polynomial generators either. Thus $Ker\varphi =0$ by Theorem \ref{Ideal}. 
\end{proof}

\begin{remark} \label{nil}
    Another immediate consequence of Theorem \ref{Ranga} is that the nilradical of a Leavitt path algebra is zero since every ideal is generated by idempotents and non-constant polynomials in cycles, none of which are nilpotent. The Jacobson radical of a Leavitt path algebra is also zero \cite[Proposition 2.3.2]{aam17}. The nilradical and the Jacobson radical of a Leavitt path algebra with coefficients in a commutative ring is studied in \cite{db24}.
\end{remark}

\section{Quotients of Leavitt Path Algebras} 

\subsection{Nonstable K-theory and the Ideals of $L(\Gamma)$ } \label{3.1}

We stratify the set of all ideals of $L(\Gamma)$ by the poset (in fact, the lattice) of graded ideals of $L(\Gamma)$: the ideal $J$ is in the stratum corresponding to the graded ideal $I$ if $I$ is the largest graded ideal contained in $J$, that is, $I$ is generated by all homogeneous elements in $J$. The graded ideal $I$ is the smallest ideal in the stratum defined by $I$ and $I$ is the intersection of all ideals in this stratum. \\

We further stratify each stratum corresponding to the graded ideal $I$ by the poset of the subsets of the set of cycles with no exit in $\Gamma/I$. Hence the finer strata correspond to ordered pairs $(I, \beta)$ where $I$ is a graded ideal of $L(\Gamma)$ and $\beta$ is a collection of cycles with no exit in $\Gamma /I$.
The stratum $(I,\emptyset)$ is the singleton $\{I\}$.\\

Finally we partition and parametrize each stratum $\mathcal{J}(I,\beta)$ 
corresponding to $(I,\beta)$ with $\beta \neq \emptyset$ using a degree function $d$ from $\beta$ to positive integers: Let $\mathcal{J}(I,\beta,d)$ be the set of non-graded ideals $J$ of  $L(\Gamma)$ with $I$ being the largest graded ideal contained in $J$, the subset of cycles with no exit in $\Gamma/I= \Gamma/J$ that contribute canonical polynomial generators to $J$ being $\beta$ and $d(C):=deg f_C(x)>0$. The set $\mathcal{J}(I,\beta, d)$ is parametrized by $\prod \big(\mathbb{F}^{d(C)} \setminus \{0\}\big)$ where the product is over all cycles $C\in \beta$ and 
the coordinates of each factor $  \mathbb{F}^{d(C)} \setminus \{0\}$ are the coefficients $(a_1, a_2,\cdots , a_{d(C)})$ of $f_C(x)$ where $f_C(C)$ is the canonical polynomial generator corresponding to the cycle $C$ with no exit in $\Gamma /J$. \\

The {\bf ideal space} of $L(\Gamma)$ is the set of ideals of $L(\Gamma)$ equipped with the stratification and the parametrization given above. We summarize this discussion: 

\begin{theorem} \label{Ideal}
If $\Gamma$ is an arbitrary digraph and $L(\Gamma)$ is its Leavitt path algebra then there is a one-to-one correspondence between the (two-sided) non-graded ideals of $L(\Gamma)$ and quadruples $(H,S,\beta, \theta)$ where $H$ is a hereditary and saturated subset of $V$, $S$ is a subset of the breaking vertices $B_H$ of $H$, $\beta$ is a nonempty collection of (geometric) cycles with no exit in $\Gamma/(H,S^H)$ and $\theta$ is a function from $\beta$ to polynomials $f(x)\in \mathbb{F}[x]$ of positive degree with $f(0)=1$.
    \end{theorem}
    \begin{proof}
     Given a non-graded ideal $J$ of $L(\Gamma)$ we define \\
     (i) $H=J\cap V$,\\
     (ii)  $S=S_J: =\{v \in B_H \> \vert \> v^H\in J\}$,\\      (iii) $\beta$ as the nonempty collection of (exclusive, geometric) cycles in $\Gamma$ such that $J\cap \mathbb{F}[C]$ is a proper nonzero ideal of $\mathbb{F}[C] \cong \mathbb{F}[x]$. \\
     (iv) $\theta(C)= f_C(x)$ where $f_C(x)\in \mathbb{F}[x]$ is the unique generator with $f_C(0)=1$ of the ideal of $\mathbb{F}[x]$ corresponding to $J\cap \mathbb{F}[C]$ in $\mathbb{F}[C]$. \\

     Conversely, given $(H,S,\beta, \theta)$ we define $J$ to be the ideal generated by $H$, $S^H$ and $\{f(C) \>\vert \> f(x)=\theta(C) \textit{ with } C\in \beta \}$. The discussion above the theorem explains that these assignments give a one-to-one correspondence. 
     \end{proof}\\

The stratification of the space of ideals of $L(\Gamma)$ and parametrization of the strata given in Theorem \ref{Ideal} are in terms of  the standard grading on $L(\Gamma)$ and the digraph $\Gamma$. In fact this stratification and parametrization  are independent of the grading  
on $L(\Gamma)$ and the digraph $\Gamma$.  
 To see this we will utilize the nonstable K-theory of $L(\Gamma)$ and a characterization of finitely generated indecomposable projective modules of Leavitt path algebras. \\ 

The nonstable K-theory $\mathcal{V}(A)$ of an algebra (or a ring) $A$ is the additive monoid of isomorphism classes of finitely generated projective $A$-modules under direct sum. There is a pre-order $\twoheadrightarrow$ on $\mathcal{V}(A)$ where $[P] \twoheadrightarrow [Q]$ if and only if $[P]=[Q] +[Q']$ for some $[Q']$ in $\mathcal{V}(A)$, equivalently 
there is an epimorphism from $P$ to $Q$. This pre-order is compatible with addition: if $[P] \twoheadrightarrow   [Q]$ then  $[P]+ [P'] \twoheadrightarrow [Q]+ [P']$. A submonoid $S$ of $\mathcal{V}(A)$ is \textbf{closed} if $[P] \in S$ and $[P] \twoheadrightarrow [Q] $ implies that $[Q] \in S$. (We prefer the term \textit{closed} to the more common \textit{order-ideal} since these submonoids are given by a Kuratowski closure operator of a Galois connection.) Arbitrary intersections of closed submonoids are closed submonoids. The closed submonoid generated by $X \subseteq \mathcal{V}(A)$, denoted by $\overline{X}$, is the intersection of all closed submonoids of $\mathcal{V}(A)$ 
containing $X$, that is, the smallest closed submonoid containing $X$.\\
 
We say that $[P] \in \mathcal{V}(A)$ is {\bf orthogonal} to $J \vartriangleleft A$, denoted $[P] \perp J$, if and only if $Hom^A(P,A/J)=0$. We have a Galois connection between  the subsets of $\mathcal{V}(A)$ ordered by inclusion and the ideals of $A$ ordered by reverse inclusion:
$$\{ J \vartriangleleft A \} \stackrel{\varphi}{\longrightarrow}  2^{\mathcal{V}(A)} \quad   \textit{
and} \qquad  2^{\mathcal{V}(A)}  \stackrel{\psi}{\longrightarrow} \{ J \vartriangleleft A \}$$
where $\varphi(J) := \{[P] \> \vert \> [P] \perp J \}$ and $ \psi(X) : = \cap \{I \vartriangleleft A \> \vert \> [P] \perp I \textit{ for all } [P] \in X \}$.

\begin{theorem} \label{Gal}
    The Galois connection defined above gives a Galois correspondence between the graded ideals of $L(\Gamma)$ and the closed submonoids of $\mathcal{V}(L(\Gamma))$, which is a lattice isomorphism when both are ordered by inclusion. 
\end{theorem}

\begin{proof} By \cite[Theorem 7]{ko}, $(\varphi, \psi)$ is a Galois connection and $\varphi (J)$ is a closed submonoid of $\mathcal{V}(L(\Gamma))$. We need to show that the image of $\varphi$ contains all closed submonoids of $\mathcal{V}(L(\Gamma))$ and $\psi(2^{\mathcal{V}(L(\Gamma))})$ is the collection of the graded ideals of $L(\Gamma)$. We will show that each closed submonoid $X$ of $\mathcal{V}(L(\Gamma))$ is uniquely and completely specified by an admissible pair 
$(H,S)$ with $(H,S^H)$  generating the ideal $\psi (X)$. This will prove the result since we have a one-to-one correspondence between admissible pairs and graded ideals of $L(\Gamma)$ by Theorem \ref{5.7}. \\

Let $H_X:= \{ v \in V \> \vert \> [vL(\Gamma)] \in X\}$ 
for a closed submonoid $X$ of $\mathcal{V}(L(\Gamma))$. The set $H_X$ is hereditary and saturated since $seL(\Gamma) \twoheadrightarrow teL(\Gamma)$ for all $e \in E$ and $vL(\Gamma)\cong \oplus_{se=v} teL(\Gamma)$ for all $v \in V$ with $0 < \vert s^{-1}(v) \vert < \infty$ by Fact \ref{iso}(ii) and (iii). Recall that the set of breaking vertices $B_{H_X}$ is the set of infinite emitters $u \in V$ with $s^{-1}(u)\cap t^{-1}(V \setminus H_X)$ finite and nonempty. Let $S_X:= \{ u \in B_{H_X} \> \vert \>  [u^{H_X} L(\Gamma)] \in X \}$ where 
$u^{H_X}= u-\sum ee^*$ with the sum over $\{ e \in s^{-1}(u)\> \vert \> te \notin H_X\}$. Thus $\{[vL(\Gamma)]\}_{v \in H_X}$ and $\{ [u^{H_X} L(\Gamma)] \}_{ u \in S_X }$ are subsets of $X$. \\

Let $Y$ be the smallest closed submonoid of $\mathcal{V}(L(\Gamma))$ containing $\{[vL(\Gamma)]\}_{v \in H_X}$ and $\{ [u^{H_X} L(\Gamma)] \}_{ u \in S_X }$. By definition $Y\subseteq X$. To see $X\subseteq Y$ first we note that every finitely generated projective $L(\Gamma)$-module is a finite direct sum of modules isomorphic to $vL(\Gamma)$ and $u_ZL(\Gamma)$ where $u_Z=u-\sum_{e\in Z} ee^*$ for an infinite emitter $u$ and a nonempty finite subset $Z$ of $s^{-1}(u)$ by Theorem \ref{3.2.11}. Also if $[P] \in X$ then $P$ is isomorphic to a finite direct sum of such modules and each summand is in $X$ because $X$ is closed. Hence it suffices to show that if $[vL(\Gamma)] \in X$ then $[vL(\Gamma)] \in Y$ and if $[u_ZL(\Gamma)] \in X$ then $[u_ZL(\Gamma)] \in Y$.\\

If $[vL(\Gamma)] \in X $ then $v \in H_X$ by definition, hence $[vL(\Gamma)] \in Y$. When $[u_ZL(\Gamma)] \in X$ then there are two cases:  $u \in H_X$ and 
$u \notin H_X$. If $u \in H_X$ then $uL(\Gamma) \twoheadrightarrow u_ZL(\Gamma) $ because $u L(\Gamma) \cong u_ZL(\Gamma) \oplus  (\sum_{e\in Z} ee^*)L(\Gamma)$ since $u_Z$ and $\sum_{e\in Z} ee^*$ are orthogonal idempotents with $u_Z + \sum ee^* =u$. Thus $[u_ZL(\Gamma)] \in Y$ because $Y$ is closed. \\

 When $u \notin H_X$, if $f \in s^{-1}(u)\setminus Z$ then $u_ZL(\Gamma)\twoheadrightarrow ff^*L(\Gamma)$ because $u_ZL(\Gamma)\cong ff^*L(\Gamma) \oplus (u_Z -ff^*)L(\Gamma)$ since $ff^*$ and $u_Z-ff^*$ are orthogonal idempotents with
$u_Z=ff^* +u_Z -ff^*$. Since $[u_ZL(\Gamma)] \in X$ and $X$ is closed we have $[ff^*L(\Gamma)] \in X$. Also  $[tfL(\Gamma)]=[ff^*L(\Gamma)]$ by Fact \ref{iso}(iv) so 
$tf \in H_X$. Hence $\{ e\in s^{-1}(u) \> \vert \> te \notin H_X\} \subseteq Z$ because $tf \in H_X$ for each $f \in s^{-1}(u) \setminus Z$. If $te \in H_X$ for all $e \in s^{-1}(u)$ then $[uL(\Gamma)]= [u_ZL(\Gamma)] + \sum_{e\in Z} [teL(\Gamma)]$ because $\{u_Z\} \cup \{ee^*\}_{e\in Z}$ is a set of orthogonal idempotents (using Fact \ref{iso}(iv) again). Since $X$ is  a submonoid of $\mathcal{V}(L(\Gamma))$ and $[u_ZL(\Gamma)] \in X$, $[teL(\Gamma)] \in X$ for all $e \in Z$, we have $[uL(\Gamma)] \in X$. So $u \in H_X$ and $[uL(\Gamma)] \in Y$. Moreover $[u_ZL(\Gamma)] \in Y$ because $uL(\Gamma) \twoheadrightarrow u_ZL(\Gamma)$. Finally, if $s^{-1}(u)\setminus H_X$ is nonempty then $u \in B_{H_X}$ and $u^{H_X}=u_Z + \sum_{g \in Z\setminus S_X} gg^*$ because $\{e \in s^{-1}(u) \> \vert \> te\notin H_X\} \subseteq Z$. Hence  $u^{H_X}L(\Gamma) \twoheadrightarrow u_ZL((\Gamma)$ because $u_Z$ and $\sum_{g \in Z\setminus S_X} gg^*$ are orthogonal idempotents. Thus $[u_ZL(\Gamma)] \in Y$ and so $X=Y$.  \\

We have shown that every closed submonoid $X$ of $\mathcal{V}(L(\Gamma))$ is the smallest submonoid containing $\{[vL(\Gamma)]\}_{v \in H_X}$ and $\{ [u^{H_X}L(\Gamma)]\}_{u \in S_X}$. Let $$P=P_X:= \Big(\bigoplus_{v \in H_X} vL(\Gamma)\Big)\oplus \Big( \bigoplus_{u \in S_X} u^{H_X}L(\Gamma) \Big).$$
By \cite[Theorem 7]{ko}  $X=P_\twoheadrightarrow$ and $X^\perp =P^\perp$. By \cite[Proposition 6(vi)]{ko} $P^\perp$ is the set of ideals containing $H_X \cup \{u^{H_X} \}_{u \in S_X}$. Hence $\psi(X)$ is the  ideal generated by $H_X \cup \{u^{H_X} \}_{u \in S_X}$, in particular $\psi(X)$ is a graded ideal. Conversely, if $H \subseteq V$ is hereditary saturated and $S \subseteq B_H$ then the ideal $I$ generated by $H \cup \{u^H \}_{u \in S}$ satisfies $H=I\cap V$ and $S=S_I:= \{ u \in B_H\> \vert \> u^H \in I\}$. Every graded ideal is of this form and we have a one-to-one correspondence between graded ideals and such pairs $(H,S)$ by Theorem \ref{5.7}. Thus the image of $\psi$ is the set of graded ideals of $L(\Gamma)$. \\

When $I$ is a graded ideal, $[vL(\Gamma)] \in \varphi (I)={^\perp I}$ if and only if $v \in I$ by \cite[Proposition 6(v)]{ko}, hence $H_{\varphi(I)} =I\cap V$. Similarly $[u^{I\cap V}L(\Gamma)] \in \varphi(I)$ if and only if $u^{I\cap V} \in I$, hence $S_{\varphi(I)}=S_I$. It follows that $\psi\varphi(I)= I$ for every graded ideal $I$ of $L(\Gamma)$ and $\varphi \psi (X)=X$ for all closed submonoids of $\mathcal{V}(L(\Gamma))$. Hence the image of $\varphi$ is the collection of closed submonoid of $\mathcal{V}(L(\Gamma))$ and we are done.   
\end{proof}\\

The monoid $\mathcal{V}(L(\Gamma))$ and its collection of closed submonoids are determined by the module category of the algebra $L(\Gamma)$. So Theorem \ref{Gal} has the consequence that the lattice of graded ideals of $L(\Gamma)$ is determined without mentioning the standard $\mathbb{Z}$-grading on $L(\Gamma)$. This was shown in \cite[Theorem 5.3]{amp07} for row-finite digraphs via the lattice isomorphism between graded ideals and hereditary saturated subsets of vertices. For arbitrary digraphs a proof (using several earlier results) is given in \cite[Theorem 3.6.23(i)]{aam17}. We give a different self-contained proof below, using a technique that will also be employed later.

\begin{corollary} \label{12}
The lattice of graded ideals of a Leavitt path algebra is a Morita invariant.
\end{corollary}

\begin{proof}
    Since finitely generated and projective can be defined in terms of the category of modules, $\mathcal{V}(L(\Gamma))$ is a Morita invariant. The partial order $\twoheadrightarrow$ and thus the lattice of closed submonoids are defined via the monoid structure of $\mathcal{V}(L(\Gamma))$. Hence the lattice of graded ideals of $L(\Gamma)$ which is isomorphic to the poset of closed submonoids of $\mathcal{V}(\Gamma)$ is also Morita invariant by Theorem \ref{Gal}. 
\end{proof}

\begin{remark}
   A precursor of Corollary \ref{12} is Proposition 2.3. in \cite{ko1} stating that there is an $F_E$-grading of $L(\Gamma)$, where $F_E$ is a free group on the set $E$ of arrows, which is universal among all gradings of $L(\Gamma)$ with all vertices and all arrows in $\Gamma$ being homogeneous elements. Every (standard) $\mathbb{Z}$-graded ideal of $L(\Gamma)$ is generated by some  vertices and  elements $u^H$ for some breaking vertices $u$ by Theorem \ref{5.7}(i). Since these generators are homogeneous with respect to the universal $F_E$-grading, if $I$ is a graded ideal with respect to the standard $\mathbb{Z}$-grading then $I$ is a graded ideal 
 with respect to every grading with $V\sqcup E$ homogeneous.  
   \end{remark}

The following is a corollary of the proof of Theorem \ref{Ideal}. It will be used to describe the stratification and the parametrization of the \textit{ideal space} of $L(\Gamma)$ in terms of the module category of $L(\Gamma)$.

\begin{corollary} \label{15}
   Let $I$ be a graded ideal of $L(\Gamma)$ and let $X={^\perp I}$, the corresponding closed submonoid of $\mathcal{V}(L(\Gamma))$. The following conditions on an $L(\Gamma)$-module $M$ are equivalent.

   \indent
   (i) $I\subseteq Ann(M)$.\\
   \indent
   (ii) $Hom(P,M)=0$ for all $[P] \in X$.\\
   \indent
   (iii) $Mv=0 = Mu^{H_I}$ for all $v \in H_I=I\cap V=H_X$ and for all $u \in S_I=S_X$.\\
   \indent (iv) $Hom(P_X, M)=0$ where $P_X:= \Big(\bigoplus_{v \in H_X} vL(\Gamma)\Big)\oplus \Big( \bigoplus_{u \in S_X} u^{H_X}L(\Gamma) \Big)$.   

   \noindent
   The full subcategory $\mathfrak{M}^X$ of $Mod_{L(\Gamma)}$ whose modules satisfy any of the equivalent conditions above is a Serre subcategory isomorphic to $Mod_{L(\Gamma/I)}$. The intersection of the ideals $J$  with $L(\Gamma)/J$  in  $\mathfrak{M}^X$ is $I$.
\end{corollary}

\begin{proof}
    Since $I$ is generated by $H_
    I=I\cap V$ and $\{u^{H_I} \>\vert \> u \in S_I \}$ by Theorem \ref{Ideal} we see that $(i)$ is equivalent to $(iii)$. By  the proof of Theorem \ref{Ideal}
    $H_X=H_I$ and $S_X=S_I$. Hence, 
as in the proof of Theorem \ref{Ideal}, $Hom(P_X, M)=0$ if and only if (ii) is satisfied. By \cite[Lemma 4(ii)]{ko} $Hom(vL(\Gamma), M) \cong Mv$ for all $v\in H_I$ and $Hom (u^{H_I}L(\Gamma),M)=Mu^{H_I}$ for all $u \in S_I$. Therefore (i) and (ii) are also equivalent. \\

The full subcategory $\mathfrak{M}^X$ of modules satisfying (i), that is, the subcategory of $L(\Gamma)$-modules annhilated by $I$, contains the $0$ module and $\mathfrak{M}^X$ is closed under isomorphism, quotients and submodules. Since $I$ is generated by idempotents, $\mathfrak{M}^X$ is also closed under extensions. Hence $\mathfrak{M}^X$ is a Serre subcategory of $Mod_{L(\Gamma)}$. 
Modules annihilated by $I$ are naturally identified with $L(\Gamma)/I$-modules  and $L(\Gamma)/I \cong L(\Gamma/I)$ by Theorem \ref{5.7}(ii). By Theorem \ref{Ideal} and (ii) above 
the intersection of the ideals $J$  with $L(\Gamma)/J$  in  $\mathfrak{M}^X$ is $\psi(X)= I$.
\end{proof}

\begin{remark}
    If $I$ is an ideal of a ring $R$  then the full subcategory of modules annhilated by $I$ is always naturally identified with the category of $R/I$-modules. This subcategory is closed under isomorphisms, quotients and submodules, but it is closed under extensions if and only if $I=I^2$ when $R$ has local units. 
\end{remark}

\noindent
We have established order isomorphisms between $4$ posets related to $L(\Gamma)$: 

(i) Graded ideals of $L(\Gamma)$.

(ii) Closed submonoids of $\mathcal{V}(L(\Gamma))$.

(iii) Serre subcategories $\mathfrak{M}^X$ of $Mod_{L(\Gamma)}$ indexed by the closed submonoids of $\mathcal{V}(L(\Gamma))$. 

(iv) Pairs $(H,S)$ with $H$ a hereditary saturated subset of $V$ and $S$ a subset of $B_H$, the breaking vertices of $H$.\\

\begin{fact}
    Let $\preceq$ be the partial order on pairs $\{(H,S)\}$ defined as $(H,S) \preceq (K,T)$ if $H\subseteq K$ and $H\cup S \subseteq K\cup T$ where $H$ is a hereditary saturated subset of $V$ and $S \subseteq B_H$. In the one-to-one correspondence of Theorem \ref{5.7} between graded ideals of $L(\Gamma)$ and such pairs $(H,S)$ given by 
     $(H\cup S^H) \leftrightarrow (H,S)$ 
inclusion of (graded) ideals corresponds to $\preceq$.
   \end{fact}
\begin{proof}
 Assume that $(H,S) \preceq (K,T)$, so $H\subseteq K$ and $H\cup S \subseteq K\cup T$. If $u \in H$ then $u \in  (K\cup T^K)$ since $H\subseteq K$. If $u \in S$ and $u \in K$ then $u^H=u-\sum e(te)e^* \in (K\cup T^K)$ since $se=u \in K$ and so $te \in K$ because $K$ is hereditary. If $u \in S$ and $u \notin K$ then $u^H =u^K-\sum f(tf)f^* \in (K \cup T^K)$ where the sum is over $ s^{-1}(u)\cap (K\setminus H)$. Hence, if $(H,S)\preceq (K,T)$ then $(H\cup S^H) \subseteq (K\cup T^K)$.\\

 If $(H\cup S^H) \subseteq (K\cup T^K)$ then $H= (H\cup S^H)\cap V \subseteq (K\cup T^K) \cap V =K$ by Thereom \ref{5.7}. If $u \in S$ then $u$ is an infinite emitter and  $s^{-1}(u) \cap t^{-1}(V\setminus K) $ is finite since it is contained in $s^{-1}(u) \cap t^{-1}(V\setminus H)$. If $s^{-1}(u) \cap t^{-1}(V \setminus K)$ is not empty then $u\in B_K$ and $u^K=u^H +\sum e(te)e^* \in (K\cup T^K)$ where the sum is over $s^{-1}(u) \cap t^{-1}(K\setminus H)$. Hence  $u\in T$. If $s^{-1}(u)\cap t^{-1}(V\setminus K)$ is empty, i.e., $t(s^{-1}(u))\subseteq K$ then $u=u^H +\sum e(te)e^* \in (K\cup T^K)\cap V=K$ since $u^H \in S^H \subseteq (K\cup T^K)$ and  
 $te \in K$ for all $e\in s^{-1}(u)$. Therefore $H\cup S \subseteq K\cup T$ in all cases.
 \end{proof}\\

Since the poset of graded ideals is a lattice so are all other posets order isomorphic to it. However, the lattice $\{ \mathfrak{M}^X\}$ of Serre subcategories of $Mod_{L(\Gamma)}$ indexed by closed submonoids of $\mathcal{V}(L(\Gamma))$ is \q{upside down} because $X \mapsto \mathfrak{M}^X$ is order reversing. These 4 lattices are Morita invariants of $L(\Gamma)$ because  closed submonoids of $\mathcal{V}(L(\Gamma))$ are Morita invariants. In the lattice of graded ideals \q{meet} is intersection and \q{join} is sum, however the join of closed submonoids $X$ and $Y$ in the lattice of closed submonoids of $\mathcal{V}(L(\Gamma))$ is not $X+Y$ in general, but $\overline{X+Y}$, the closure of $X+Y$. That is, the lattice of closed submonoids of $\mathcal{V}(L(\Gamma))$ is not a sublattice of all submonoids of $\mathcal{V}(L(\Gamma))$ in general.\\

We are now able to describe the strata of the space of ideals of $L(\Gamma)$ in terms of the closed submonoids of $\mathcal{V}(L(\Gamma))$: If $J$ is an ideal of  $L(\Gamma)$ then the largest graded ideal contained in $J$ is $\psi \varphi (J)$ and its stratum is $\varphi^{-1}(\varphi (J))$, in the notation of Theorem \ref{Gal}. To describe the finer strata $\mathcal{J}(I,\beta)$ and their parametrization algebraically, independent of the digraph $\Gamma$, we will need a characterization of finitely generated indecomposable projective (henceforth denoted as {\bf fgip}) modules of Leavitt path algebras.

\subsection{Finitely Generated Indecomposable Projectives} \label{3.2}

The finer strata $\mathcal{J}(I,\beta)$ of the space of ideals of $L(\Gamma)$ were defined in terms of the cycles with no exit in $\Gamma/I$.
We will now establish a one-to-one correspondence between such cycles and isomorphism classes of non-simple fgips of $L(\Gamma/I)$.

\begin{lemma} \label{baz}
(i) If $v$ is a sink in $\Gamma$ then $X:=\{q^*  \> \vert \> q \in Path(\Gamma)\>, \> tq=v \}$ is an $\mathbb{F}$-basis for the projective $L(\Gamma)$-module $vL(\Gamma)$.\\
(ii) When $v$ is a leak, let $p_0=v$ and 
$p_n=e_1e_2 \cdots e_n$ be the unique path of length $n$ with $sp_n=v$. Then 
$X:=\{ p_nq^* \> \vert \>  n \in \mathbb{N} ,\> 
q \in Path(\Gamma) , \> tq=te_n, \> q\neq q'e_n \}$ is an $\mathbb{F}$-basis for $vL(\Gamma)$.\\
(iii) If $v$ is a sink or a leak in $\Gamma$ then the projective $L(\Gamma)$-module $vL(\Gamma)$ is simple.\\
 (iv) If $C$ is a cycle with no exit and $v=sC$ then  the projective module $vL(\Gamma)$ is an $({\mathbb{F}}[x,x^{-1}] ,L(\Gamma))$-bimodule and free as a $vL(\Gamma)v \cong {\mathbb{F}}[x,x^{-1}]$-module with basis $P_C^*$ where $P_C=\{ p 
\in Path(\Gamma) \> \vert \> tp=v \textit{ and } p\neq qC \}$.
\end{lemma}
 \begin{proof}
 (i) The set $\{pq^* \> \vert \> sp=v \>, tp=tq \> \}$  spans $vL(\Gamma)$. If $v$ is a sink then $p=v$. If  $q^*r\neq 0$ with $tq=v=tr$ then $q$ is an initial segment of $r$ or $r$ is an initial segment of $q$ by Fact \ref{p*q} and $q=r$ since $v$ is a sink. If $\lambda_1 q_1^* + \lambda_2q_2^* + \cdots + \lambda_nq_n^*=0$  with distinct $q_i^* \in X$ and $\lambda_i \in \mathbb{F}$ for $i=1,2, \dots , n$ then $(\lambda_1 q_1^* + \lambda_2q_2^* + \cdots + \lambda_nq_n^*)q_j=\lambda_jv=0$.  Proposition \ref{nonzero}(i) implies that $\lambda_j=0$  for each $j$, hence $X$ is an $\mathbb{F}$-basis for $vL(\Gamma)$.\\

 (ii) When $v$ is a leak $\{p_nq^* \vert \> n \in \mathbb{N},\> q \in Path(\Gamma), \> tq=tp_n \}$ spans $vL(\Gamma)$ since there is a unique path $p_n$ of length $n$ with $sp_n=v$. if $q=q'e_n$ then $p_nq^*=p_{n-1}q'^*$ because $se_n=e_ne_n^*$ by (CK2). Repeating this if necessary we see that $X$ spans $vL(\Gamma)$. If $p_mq^*(p_nr^*)^* =p_mq^*rp_n^*
 \neq 0$ with $p_mq^*, \> p_nr^* \in X$ then $q$ is an initial segment of $r$ or $r$ is an initial segment of $q$ by Fact \ref{p*q}. If $q$ is an initial segment of $r$ then $r=qr'$ and $p_mq^*rp_n^*=p_mr'p_n^*$. Since $tr' =tr=tp_n$ and $tp_m=tq=sr'$ we get   $p_mr'=p_n$ because $v$ is a leak. Then $r' = tp_n$, otherwise the last arrow of $r$ would be $e_n$ contradicting the definition of $X$. Hence $q=r$ and $m=n$ (because $te_m \neq te_n$ if $m=n$ as $v$ is a leak). Similarly, if $r$ is an initial segment of $q$ implies $q=r$ and $m=n$. Therefore $xy^*\neq 0$ with $x, \>y \in X$ implies that $x=y$. If $\lambda_1 x_1 + \lambda_2x_2 + \cdots + \lambda_nx_n=0$ with distinct $x_i\in X$ and $\lambda_i \in \mathbb{F}$ for $i=1,2, \dots , n$ then $(\lambda_1 x_1 + \lambda_2x_2 + \cdots + \lambda_nx_n)x_j^*=\lambda_jv=0$.  Proposition \ref{nonzero}(i) implies that $\lambda_j=0$  for each $j$, hence $X$ is an $\mathbb{F}$-basis for $vL(\Gamma)$.\\

 (iii) We have seen above that
 $xx^*=v$ and if $x\neq y$ then 
 $xy^* =0 $ for all $x, y$ in the corresponding basis $X$. If $0\neq \lambda_1x_1+\lambda_2x_2 +\cdots +\lambda_nx_n \in vL(\Gamma)$ with $\lambda_1\neq 0$ and $x_i$ distinct elements of $X$ then $(\lambda_1x_1+\lambda_2x_2 +\cdots +\lambda_nx_n)\lambda_1^{-1}x_1^*=v $. Since $v$ generates $vL(\Gamma)$ the only nonzero $L(\Gamma)$-submodule of $vL(\Gamma)$ is $vL(\Gamma)$. Hence $vL(\Gamma)$ is simple.\\

 (iv) Let $v=sC$, so $vL(\Gamma)$ 
is a $(vL(\Gamma)v ,L(\Gamma))$-bimodule and $vL(\Gamma)v \cong \mathbb{F}[x,x^{-1}]$ by Lemma \ref{2.2}(ii). 
If $pq^* \in L(\Gamma)$ with $sp=v$ and $tp=tq$ then $tp \in V_C$ since $C$ has no exit. Such $pq^*$ spans $vL(\Gamma)$. By repeated applications of (CK2) as needed we may assume that $tp=sC=tq$ (again since $C$ has no exit). Using $CC^*=v$ we may now express such $pq^*$ as $C^nr^*$ with $r \in P_C$ and $n \in \mathbb{Z}$. Hence 
$\{C^np^* \> \vert \> n \in \mathbb{Z}\textit{ and } \> p \in P_C \}$ spans $vL(\Gamma)$. Thus any element in $vL(\Gamma)$ can be expressed as $\sum_{i=1}^m \alpha_i p_i^*$ with $\alpha_i \in vL(\Gamma)v \cong \mathbb{F}[x,x^{-1}]$ and $p_i$ distinct elements of $P_C$. We can recover $\alpha_j$  as $(\sum_{i=1}^m \alpha_i p_i^*)p_j$ for all $j$
by Fact \ref{p*q} 
since distinct $p_i \in P_C$ can not be initial segments of one another. Therefore $vL(\Gamma)$ is a free $\mathbb{F}[x,x^{-1}]$-module with basis $\{p^* \> \vert \> p \in P_C \}$.
\end{proof}\\

The proposition below will be used in the proof of Theorem \ref{26} but it should also be of independent interest. 

\begin{proposition} \label{basit}
  The following are equivalent for an $L(\Gamma)$-module $M$.\\
  (i) $M$ is simple and projective.\\
  (ii) $M\cong vL(\Gamma)$ where $v$ is a sink or a leak.\\
  (iii) $M$ is simple and its support $V_{_M}$ contains a sink or a leak.\\
  (iv) $M$ is finitely generated  projective and $End^{L(\Gamma)}(M) \cong \mathbb{F}$.\\
If $M$ satisfies any of these equivalent conditions then $V_{_M}$ contains either a unique sink or a unique equivalence class of a leak. We have a one-to-one correspondence between the set of isomorphism classes of simple projective $L(\Gamma)$-modules and the union of the set of sinks and the set of equivalence classes of leaks.
\end{proposition}
\begin{proof}
(ii) $\Rightarrow$ (i) This follows from Lemma \ref{baz} above.\\ 

(i) $\Rightarrow$ (ii) If $M$ is simple and projective then $M$ is finitely generated, hence $M\cong v_1L(\Gamma)\oplus v_2L(\Gamma)\oplus \cdots \oplus v_nL(\Gamma) $ by Theorem \ref{3.2.11}. Since $0\neq v_i=v_i^2 \in v_iL(\Gamma)$ and $M$ is simple there can be only one summand, that is, $M\cong vL(\Gamma)$ for some vertex $v$. If $v$ is not a sink then $vL(\Gamma) \cong \bigoplus_{se=v} te_iL(\Gamma)$ hence $s^{-1}(v)  =\{e \}$. We may replace $v$ with $te$ and continue this process. There are 3 possibilities: 1) The vertex $v$ is  leak and we are done. 2) After finitely many steps we reach a sink $w$ so $M\cong vL(\Gamma) \cong wL(\Gamma)$ and we are done. 3) After finitely many steps we reach a cycle $C$ with no exit, so $M\cong vL(\Gamma) \cong sCL(\Gamma)$. But there is an epimorphism from $sCL(\Gamma)$ to the Chen module of linear combinations of infinite paths tail equivalent to $C^\infty$ \cite{che15} given by left multiplication with $C^\infty$. (The action of $a \in vL(\Gamma)$ on an infinite path $\alpha$ is $a^*\alpha$.) The kernel of this epimorphism contains $0\neq v-C \in  \mathbb{F}\Gamma \subseteq L(\Gamma)$, therefore $sCL(\Gamma)$ is not simple so this case can not happen.\\

(ii) $\Rightarrow$ (iii) If $M\cong vL(\Gamma)$ where $v$ is a sink or a leak then $vL(\Gamma)$ is simple by Lemma \ref{baz} and $v=v^3\in vL(\Gamma)v$. Hence  $ 0\neq v \in Mv$ and  $ v\in V_M$.\\

(iii) $\Rightarrow $ (ii) If $M$ is simple and $v\in V_M$ with $v$ a sink or a leak then there is a nonzero $m \in Mv$. The $L(\Gamma)$-module homomorphism $m  \underline{\>\>} : vL(\Gamma)\longrightarrow M$ given by left multiplication with $m$ has $mv=m \neq 0$ in its image. The homomorphism $m \underline{\>\>} $ is onto since $M$ is simple, $m \underline{\>\>}$ is one-to-one because $vL(\Gamma)$ is simple by Lemma \ref{baz}. Hence  $M \cong vL(\Gamma)$.\\

(ii) $\Rightarrow$ (iv): $vL(\Gamma)$ is generated by $v=v^2$ and it is projective by \cite[Lemma 4(i)]{ko}. Also $End^{L(\Gamma)}(vL(\Gamma))\cong vL(\Gamma)v$ by \cite[Lemma 4(iii)]{ko}. The corner algebra $vL(\Gamma)v$ is generated by $\{ pq^* \> \vert \> sp=v=sq, \> tp=tq \}$. If $v$ is a sink then $p=v=q$, hence $vL(\Gamma)v= \mathbb{F}v \cong \mathbb{F}$.  If $v$ is a leak then $p=q$ because there is a unique path from $v$ to $tp=tq$. By repeated applications of (CK2) $pq^*=v$, hence $vL(\Gamma)v=\mathbb{F}v \cong \mathbb{F}$.  \\

(iv) $\Rightarrow$ (ii) If $M$
is a finitely generated projective $L(\Gamma)$-module then $M $ is isomorphic to a finite direct sum of modules of the type $vL(\Gamma)$ and $v_ZL(\Gamma)$ by Theorem \ref{3.2.11}. But $End^{L(\Gamma)}(vL(\Gamma)) \neq 0\neq End^{L(\Gamma)}(v_ZL(\Gamma))$
since $0 \neq v \in vL(\Gamma)v \cong End^{L(\Gamma)}(vL(\Gamma)) $ and $0 \neq v_Z \in v_ZL(\Gamma)v_Z \cong End^{L(\Gamma)}(v_ZL(\Gamma)) $ by \cite[Lemma 4(iii)]{ko}, hence there can be only one summand because $End (A) \bigoplus End (B)$ is isomorphic to a subalgebra of $End (A\bigoplus B)$ and $dim End^{L(\Gamma)}(vL(\Gamma))=1$.\\

$End^{L(\Gamma)}(v_ZL(\Gamma)) \cong v_ZL(\Gamma)v_Z$ is infinite dimensional for all $v_Z$ since it contains the infinitely many nonzero orthogonal idempotents $\{ ee^* \> \vert \> \> e \in s^{-1}(v) \setminus Z\}$. Hence $M \cong vL(\Gamma)$ for some $v \in V$. \\

If $M\cong vL(\Gamma)$ and there is a path from $v$ to a cycle $C$  then $sCL(\Gamma)$ is a summand of $vL(\Gamma)$ by Fact \ref{iso}(i). But $End^{L(\Gamma)} (sCL(\Gamma))\cong sCL(\Gamma)sC$ contains the infinite, linearly independent set $\{C, \> C^2, \cdots \} \subseteq \mathbb{F}\Gamma \subseteq L(\Gamma)$. Hence there are no cycles in $V_{v \> \leadsto}$, the successors of $v$.\\

If $v$ is not a sink then 
$vL(\Gamma) \cong \bigoplus_{se=v} teL(\Gamma)$ hence $s^{-1}(v)=\{e\}$, a singleton and $vL(\Gamma) \cong teL(\Gamma)$. If $v$ is not a leak then after finitely many steps we reach a sink $w$ with $M\cong wL(\Gamma)$. Thus $M\cong vL(\Gamma)$ with $v$ either a sink or a leak.\\

Now we can establish the one-to-one correspondence between isomorphism classes of simple projective $L(\Gamma)$-modules and the union of the set of sinks and the set of equivalence classes of leaks in $\Gamma$:  If $M\cong vL(\Gamma)$ for some sink $v$ and $w$ is a sink in $V_M$ then there is $pq^* \in L(\Gamma)$ with $sp=v$, $tp=tq$ and $sq=w$. Since $v$ and $w$ are sinks $v=p=tp=tq=q=w$, so $v$
is the unique sink in $V_M$. If $u \in V_M$ is a leak then there is $pq^*$ in $L(\Gamma)$ with $sp=v$, $tp=tq$ and $sq=u$. Now $v=p=tp=tq=q=u$ gives a contradiction because there can not be a path $q$ from a leak to a sink. Hence there is no leak in $V_M$. \\

Similarly, if $M\cong uL(\Gamma)$ with $u$ a leak then there is no sink in $V_M$. If $v$ is another leak in $V_M$ then there is $pq^*$ in $L(\Gamma)$ with $sp=u$, $tp=tq$ and $sq=v$. So the leaks $u$ and $v$ are equivalent by definition. Conversely, if $u$ and $v$ are equivalent leaks then there is a $0 \neq pq^*$ in $uL(\Gamma)$ with $sp=u$, $tp=tq$ and $sq=v$. Hence $v \in V_M$ and the leaks in $M\cong uL(\Gamma)$ constitute a unique equivalence class of leaks.\\

Thus each simple projective $L(\Gamma)$-module $M$ contains a unique sink or a unique equivalence class of leaks in its support $V_M$ and this defines a bijection between isomorphism classes of simple projective $L(\Gamma)$-modules and the union of the set of sinks and the set of equivalence classes of leaks.
\end{proof}\\

A vertex $v$ in $\Gamma$ is called a \textit{line point} if it is either a leak or there is a unique path $p$ with no exit such that $sp=v$ and $tp$ is a sink. That is, $v$ is a line point if and only if $\Gamma_{v \>\leadsto}$ is either a (finite) path or an infinite path.

\begin{corollary}
    The projective $L(\Gamma)$-module $vL(\Gamma)$ is simple if and only if $v$ is a line point.
    \end{corollary}

    \begin{proof}
        If $v$ is  a leak or there is a unique path $p$ with no exit such that  $sp=v$ and $tp$ is a sink then $tpL(\Gamma) \stackrel{p \underline{\>\>}}{\longrightarrow} vL(\Gamma)$ is an isomorphism by (CK1) and (CK2). Hence $vL(\Gamma)$ is simple  in both cases by Proposition \ref{basit}.\\

        If $p$ is a path in $\Gamma$ with $sp=v$ then $tpL(\Gamma)$ is a nonzero summand of $vL(\Gamma)$ by Fact \ref{iso}(i). Therefore, if  $vL(\Gamma)$ is simple then  there is no branch point or a cycle in $\Gamma_{v \>\leadsto}$ as in the proof of Proposition \ref{basit}. Thus $v$ is a line point.        
        \end{proof}

\begin{theorem} \label{fgip}
Let $\Gamma$ be a digraph and   $P$ a fgip module  of $L(\Gamma)$. If $P$ is not simple then $P\cong vL(\Gamma)$  where $v\in V$ is on a cycle with no exit. We have a   one-to-one correspondence between the cycles in $\Gamma$ with no exit and the isomorphism classes of non-simple fgip $L(\Gamma)$-modules, which maps a cycle $C$ to the isomorphism class of $vL(\Gamma)$ where $v$ is a vertex on $C$. In the reverse direction, given a non-simple fgip $P$ the corresponding cycle is the unique cycle with no exit in the support subgraph $\Gamma_P$.   

\end{theorem}

\begin{proof}  
By Theorem \ref{3.2.11}, $P$ is isomorphic to a finite direct sum of modules of the type $vL(\Gamma)$ and $v_ZL(\Gamma)$ where 
$v_Z:=v-\sum_{e\in Z} ee^*$ with $Z$ a finite subset of the infinite set $s^{-1}(v)$. Since $P$ is indecomposable and each $v$ or $v_Z$ is nonzero by Proposition \ref{nonzero}(i), (ii),  there is only one summand. if $s^{-1}(v)$      is infinite and $e \in s^{-1}(v) \setminus Z$ then $v_ZL(\Gamma) \cong v_YL(\Gamma) \oplus ee^*L(\Gamma)$ where $Y=Z \cup \{e\}$. Hence $P\cong vL(\Gamma)$ and $s^{-1}(v)$ is finite since $v_Z=v$ when $Z = \emptyset$.
If $q$ is a path with $sq=v$ then $tqL(\Gamma)$ is isomorphic to a summand of $vL(\Gamma)$ by Fact \ref{iso}(i), hence $vL(\Gamma) \cong tqL(\Gamma)$. The vertex $tq$ is not a branch vertex by Fact \ref{iso}(iii) because $P$ is the indecomposable. Also $tq$ is not a sink or a leak by Lemma \ref{baz}(iii) since $P$ is not simple. Therefore there is a unique infinite path starting at $v$ and ending at a cycle $C$ with no exit and $P\cong uL(\Gamma)$ where $u$ is a vertex on $C$. \\

Conversely, let $v$ be on a cycle $C$ with no exit and let $M$ be a nonzero summand of $vL(\Gamma)$ with section $s: vL(\Gamma) \longrightarrow M$, so the composition $M \hookrightarrow vL(\Gamma) \stackrel{s}{\longrightarrow} M$ is $id_M$. By Lemma \ref{baz}(iv) if $0 \neq m \in M$ then $m= \sum_{i=1}^n \alpha_ip_i^*$ where $0 \neq \alpha_i \in  vL(\Gamma)v$ and $p_i \in P_C$ for $i=1, 2, \cdots , n$. Note that if $1 \leq i \neq j \leq n$ then $p_i$ is not an initial segment of $p_j$ by the definition of $P_C$. Hence $mp_j= \alpha_j \in Mv$. Therefore $Mv \neq 0$. Applying the functor $Hom(vL(\Gamma), \underline{\>\>})$ to  $M \hookrightarrow vL(\Gamma) \stackrel{s}{\longrightarrow} M$ we get that the $vL(\Gamma)v$-module $Mv$ is a summand of $vL(\Gamma)v \cong \mathbb{F}[x,x^{-1}]$ by \cite[Lemma 4(ii), (iii)]{ko}. Since $vL(\Gamma)v \cong \mathbb{F}[x,x^{-1}]$ is indecomposable as an $\mathbb{F}[x,x^{-1}]$-module, we see that $M=vL(\Gamma)$, that is, $vL(\Gamma)$ is indecomposable.\\

If $e$ is an arrow on the cycle $C$ with no exit then $teL(\Gamma) \stackrel{e\underline{\>\>}}{\longrightarrow} seL(\Gamma)$ is an isomorphism by Fact \ref{iso}(iii) since $s^{-1}(se)=\{e\}$. Hence the isomorphism type of $vL(\Gamma)$ depends only on the cycle $C$ with no exit that $v$ is on. If $u$ is on a cycle $C$ with no exit and $v$ is on a different cycle with no exit then $Hom(vL(\Gamma), \> uL(\Gamma)) \cong uL(\Gamma)v$ by \cite[Lemma 4(iv)]{ko}. Also $uL(\Gamma)v=0$ by Lemma \ref{baz}(iv) since there is no $p \in P_C$ with $sp=v$ because $v$ is on a cycle with no exit. Therefore distinct cycles with no exits corresponds to distinct isomorphism classes of fgips.\\

If $P \cong vL(\Gamma)$ where $v$ on a cycle $C$ with no exit then the support $V_P$ of $P$ is $V_{\leadsto C}$ by Lemma \ref{baz}(iv) and $C$ is the unique cycle with no exit in $\Gamma_P =\Gamma_{\leadsto C} $. Thus assigning the isomorphism class of the fgip $vL(\Gamma)$ to a cycle $C$ with no exit where $v$ is an arbitrary vertex on $C$ defines a bijection. The inverse bijection maps the isomorphism class of a  fgip $P$ to the unique cycle with no exit in its support subgraph $\Gamma_P$.
 \end{proof}\\

Now we can describe the stratification and the parametrization of the ideal space of $L(\Gamma)$ algebraically, independent of the digraph $\Gamma$. 

\begin{theorem} \label{cebirsel} 
    There is a one-to-one correspondence between the non-graded ideals of $L(\Gamma)$ and triples $(X, \beta, \theta)$ where $X$ is a closed submonoid of $\mathcal{V}(L(\Gamma))$, $\beta \neq \emptyset$ is a subset of non-simple fgip modules in $\mathfrak{M}^X$ and $\theta$ is a function from $\beta$ to  polynomials $f(x)\in \mathbb{F}[x]$ of positive degree with $f(0)=1$. If $J$ is an ideal of $L(\Gamma)$ then $X={^\perp J}$ and $\mathfrak{M}^X \cong Mod_{L(\Gamma)/I}$ where $I$ is the intersection of the ideals in $X^\perp$. The isomorphism class $[P] \in \mathcal{V}(L(\Gamma)/I)$ is in $\beta$ if and only if $P$ is a non-simple fgip $L(\Gamma)/I$-module with 
$Hom(P, L(\Gamma)/J)$  nonzero and finite dimensional. The polynomial $\theta([P])$ is the unique $f(x) \in \mathbb{F}[x]$ with $f(0) =1$ generating the annihilator of the $\mathbb{F}[x,x^{-1}]\cong End(P)$-module $Hom(P, L(\Gamma)/J)$. 
\end{theorem}

\begin{proof}
The stratum $(X, \beta)$ corresponds to the stratum $(I, \beta)$ of Theorem \ref{Ideal} via the one-to-one correspondence between graded ideals  and closed submoniods of $\mathcal{V}(L(\Gamma))$ in Theorem \ref{Gal} where $X={^\perp I}$. Corollary \ref{15} shows that $\mathfrak{M}^X \cong Mod_{L(\Gamma/I)}$. By Proposition \ref{basit} and Theorem \ref{fgip}, non-simple fgip $L(\Gamma/I)$-modules are in one-to-one correspondence with the cycles in $\Gamma/I$ with no exit where the fgip $P=vL(\Gamma/I)$ with $v$ a vertex on the cycle $C$. Also $vL(\Gamma/I)v\cong End(vL(\Gamma/I)) \cong \mathbb{F}[x,x^{-1}]$ by \cite[Lemma 4(iii)]{ko} and Lemma \ref{2.2}(ii) above. Moreover, $Hom( vL(\Gamma/I), L(\Gamma)/J) \cong (L(\Gamma)/J)v$ as $End(vL(\Gamma/I)) \cong \mathbb{F}[x,x^{-1}]$-modules by \cite[Fact 2]{ko} where $vL(\Gamma/I)v$ acts on $L(\Gamma/I)v$ by right multiplication. The canonical polynomial generator $f_C(C)$ of the ideal $J$ corresponds to the unique $f_C(x) \in \mathbb{F}[x]$ with $f_C(0)=1$ generating the annihilator of the $\mathbb{F}[x,x^{-1}]\cong End(P)$-module $Hom(P, L(\Gamma)/J)$ since $x \leftrightarrow C$ in the isomorphism between $\mathbb{F}[x,x^{-1}]$ and $vL(\Gamma/I)v$. \end{proof}\\

A module in a Serre subcategory is finitely generated or simple or indecomposable if and only if it is so in the ambient category. However, projective modules in $\mathfrak{M}^X \backsimeq Mod_{L(\Gamma)/I}$ are not necessarily projective as $L(\Gamma)$-modules. That is why we needed the Morita invariant Serre subcategories of $\mathfrak{M}^X$ and their fgips to describe stratification and the parametrization of the ideal space of $L(\Gamma)$.\\

The given filtration, the corresponding stratification and the parametrization of the strata of the ideal space of an LPA are  Morita invariants, since they can be defined categorically: The cycles in $\Gamma$ with all exits in the ideal $J$ correspond to isomorphism classes of nonsimple fgips in $\mathfrak{M}^X$. The subset $\beta$ corresponds to $[P]$ in $\mathcal{V}(L(\Gamma)) $ with $Hom (P, L(\Gamma)/J)$ finite dimensional where $P$ is a nonsimple fgip in $\mathfrak{M}^X$ and $\theta$ assigns to $[P]$ the unique $f(x) \in \mathbb{F}[x]$ with $f(0)=1$ generating the ideal $Ann(Hom(P, L(\Gamma)/J))$ of the algebra  $End(P) \cong \mathbb{F}[x,x^{-1}]$.

\begin{remark} \label{Morita}
   Given a triple $(X,\beta, \theta)$ where $X$ is a closed submonoid of $\mathcal{V}(L_{\mathbb{F}}(\Gamma))$ and  $\beta$ is a subset of isomorphism classes of nonsimple fgips in $\mathfrak{M}^X$ with $\theta$ a function from $\beta$ to nonconstant polynomials in $\mathbb{F}[x]$ whose constant term is $1$, we can recover the  corresponding ideal $J$ explicitly as follows.
We define the full subcategory $\mathfrak{M}^X_{(\beta, \theta)}$ of $\mathfrak{M}^X$ whose modules $M$ satisfy the following: $Ann(Hom(P, M))$ is generated by $\theta([P])$ for all $[P] \in \beta$. The ideal $J$ is the annihilator of  $\mathfrak{M}^X_{(\beta, \theta)}$, that is, the intersection of the annihilators of the  modules $M$ in $\mathfrak{M}^X_{(\beta, \theta)}$. (It suffices to take this intersection over the set of isomorphism classes of finitely generated modules in $\mathfrak{M}^X_{(\beta, \theta)}$.)
\end{remark}

\begin{remark}
   The strata $(X,\beta)$ can be defined for an arbitrary algebra $A$ with $X= {^\perp J}$ where $I$ and $\beta$ are defined as in Theorem \ref{cebirsel}. However, the parametrization $\theta$ of the stratum $(X,\beta)$ requires that $End(P) \cong \mathbb{F}[x,x^{-1}]$ when $P$ is a non-simple finitely generated indecomposable projective $A/I$-module. Moreover, to establish the one-to-one correspondence between the ideals of $A$ and the triples $(X, \beta, \theta)$ we need a characterization of ideals as in Theorem \ref{Ideal}. 
\end{remark}

\subsection{Quotients of Leavitt Path Algebras} \label{3.3}

When the digraph $\Gamma$ is row-finite, using the reduction algorithm \cite[Theorem 4.1]{ko2} an \textit{exclusive cycle} (i.e., a cycle disjoint from other cycles) can be replaced by a loop without changing the Morita type of the Leavitt path algebra. The lemma below shows that in an arbitrary digraph $\Gamma$ cycles with no exit can be replaced with  loops to obtain a digraph $\Lambda$ so that $L(\Gamma)\cong L(\Lambda)$. 
If $\{ \> C_i=e_iq_i \}$ where $e_i \in E$ is a collection of distinct geometric cycles with no exit in $\Gamma$ then  $\Lambda$ is defined to be the digraph obtained from $\Gamma$ by cutting each $C_i$ at $te_i$ and reattaching $te_i$ to $se_i$, that is, replacing each $e_i$ with a loop $e_i'$ where $se_i'=se_i=te_i'$. So $\Lambda$ has the same set of vertices and essentially the same set of arrows as $\Gamma$.

\begin{example} The representatives of the unique geometric cycle in the digraph $\Gamma$ are $efg$,  $\> fge$ and $gef$.
   $$ 
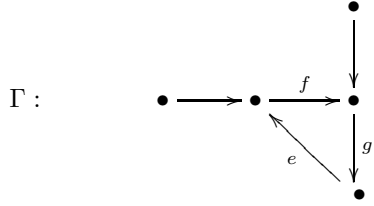
\xymatrix{ &&{\bullet} \ar[d]\\
   \Gamma: \qquad  \qquad {\bullet} \ar[r] &{\bullet}  \ar[r]^{f} & {\bullet} \ar[d]
^{g}\\
&& \> \> {\bullet} \ar[ul]^{e} }$$

\noindent
We obtain the 3 possible digraphs $\Lambda$ given below:
$$ 
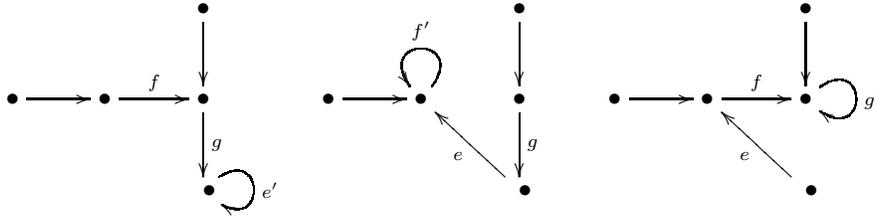
\xymatrix{ &&{\bullet} \ar[d]\\
   {\bullet} \ar[r] &{\bullet}  \ar[r]^{f} & {\bullet} \ar[d]^{g}\\
&& \> \> {\bullet}\ar@(ru,rd)^{e'} } \quad 
 \xymatrix{ &&{\bullet} \ar[d]\\
   {\bullet} \ar[r] &{\bullet}\ar@(ur,ul)_{f'}   & {\bullet} 
   \ar[d]
^{g}\\
&& \> \> {\bullet} \ar[ul]^{e} } \qquad 
\xymatrix{ &&{\bullet} \ar[d]\\
   {\bullet} \ar[r] &{\bullet}  \ar[r]^{f} & {\bullet} \ar@(ru,rd)^{g'}
\\
&& \> \> {\bullet} \ar[ul]^{e} }\\
$$

Note that the isomorphism type of the digraph $\Lambda $ depends on the choice of the initial arrow of the cycle.
\end{example}

\begin{lemma} \label{loop}
   If the digraphs $\Gamma$ and $\Lambda$ are as above then $L(\Gamma)$ is isomorphic to $L(\Lambda )$ as $*$-algebras.
\end{lemma}

\begin{proof}
    The isomorphism $\varphi :L(\Gamma)\longrightarrow  L(\Lambda)$ is the $*$-homomorphism defined by $\varphi(v) :=v$ for all $v\in V$ and $\varphi (e):=e$ for all $ e \in E$ unless $e=e_i$ for some $i$, and $\varphi( e_i):=e_i'q_i^* $ for all $i\in I$. Its inverse  $\psi :L(\Lambda)\longrightarrow  L(\Gamma)$ is the $*$-homomorphism defined by $\psi(v) :=v$ for all $v\in V$ and $\psi (e):=e$ for all $ e \in E$ unless $e=e_i'$ for some $i$, and $\psi( e_i'):=C_i$ for all $i \in I$.\\

    Both $\varphi$ and $\psi$ are identity on the vertices, so the relations (V) are satisfied. Each arrow $e$ maps to a path $p$
in the doubled digraph  such that  $se=sp$ and $te=tp$, so the relations (E) are satisfied. Under both $\varphi$ and $\psi$ the image of an arrow is not an initial segment of the image of another arrow, so $\varphi( e^*f) =0$ and $\psi(e^*f)=0$ if $e \neq f$ in either $\Gamma$ or $\Lambda$ by Fact \ref{p*q}. Also $q_iq_i^*=sq_i$ in $L(\Lambda)$ by (CK2) because $q_i$ has no exit. Hence $\varphi (e_i^*e_i)= q_ie_i'^*e_i' q_i^*=q_ite_i' q_i^*=q_itq_i  q_i^* =q_iq_i^* = sq_i= \varphi (sq_i)=\varphi(te_i) $ and the relations (CK1) are satisfied. The relations (CK2) need to be checked only at the vertices $se_i=sC_i=se_i'$ for $i \in I$, for both $\varphi$ and $\psi$. Since $C_i$ has no exit $\psi(e_i'e_i'^*)=C_iC_i^*=sC_i$ and $\varphi(e_ie_i^*)= e_i'q_i^*q_ie_i'^*=e_i'e_i'^*=se_i'$. 
\end{proof}\\

The following lemmas enable us to eliminate certain types of finitely generated projective $L(\Gamma)$-modules to achieve the classification of those with finite dimensional endomorphism algebras in Theorem \ref{26}.

\begin{lemma} \label{branch}
If the successor subgraph $\Gamma_{v \> \leadsto}$ of the vertex $v$ contains infinitely many branch vertices but neither cycles nor  infinite emitters then we can find paths $\{p_i\}_{i\in \mathbb{N}}$ in $\Gamma $ such that $sp_0=v$ and $tp_i=sp_{i+1}$ being distinct branch vertices for all $i\in \mathbb{N}$. 
\end{lemma}
\begin{proof}
If $v$ is not a branch vertex then there is a unique path $p$ with $sp=v$ and $tp$ is the branch vertex nearest to $v$ among the successors of $v$. Let $p_0=p$. Note that $\Gamma_{tp_0 \> \leadsto}$ contains exactly the same branch vertices as $\Gamma_{v \> \leadsto}$  hence infinitely many.\\\\
If $v$ is a branch vertex 
then there is an $e \in s^{-1}(v)$ with $V_{te\> \leadsto}$ containing infinitely many branch vertices 
since $V_{v\> \leadsto}= \{v \} \cup \bigcup_{se=v} V_{te \> \leadsto}$. If $te$ is a branch vertex then let $p_0=e$. If $te$ is not a branch vertex then there is a unique path $q$ with $sq=te$ and $tq$ is the branch vertex nearest to $te$ among the successors of $te$. In this case, let $p_0:=eq$. In both cases $\Gamma_{tp_0 \> \leadsto}$ contains infinitely many branch vertices but neither cycles nor infinite emitters since $\Gamma_{tp_0 \> \leadsto}$ is a subgraph of $\Gamma_{v \> \leadsto}$.\\

Replacing $v$ with $tp_0$ we can find $p_1$ as above. We can inductively define $p_{i+1}$ from $p_i$ with $\Gamma_{tp_i \> \leadsto}$ containing infinitely many branch vertices for $i \in \mathbb{N}$. Since $\Gamma_{v \> \leadsto}$ is acyclic and each $p_i$ has positive length, $tp_i$ are distinct branch vertices for all $i \in \mathbb{N}$.
\end{proof}

\begin{theorem} \label{26}
 If $P$ is a finitely generated projective $L(\Gamma)$-module then the following are equivalent:\\
 \noindent
 (i) The endomorphism algebra $End^{L(\Gamma)}(P)$ is finite dimensional;\\
 (ii) $P\cong \bigoplus_{i=1}^n v_iL(\Gamma)$ where each $v_i$ is a sink or a leak;\\
 (iii) $End^{L(\Gamma)}(P)$ is isomorphic to a finite direct sum of matrix algebras over $\mathbb{F}$. 
 \end{theorem} 
\begin{proof} $(i) \Rightarrow (ii):$ By Theorem \ref{3.2.11}, $P$ is isomorphic to a finite direct sum of modules of the type $vL(\Gamma)$ and $v_ZL(\Gamma)$ where 
$v_Z:=v-\sum_{e\in Z} ee^*$ with $Z$ a finite subset of the infinite set $s^{-1}(v)$.
 If $End^{L_{\mathbb{F}}(\Gamma)}(P)$ is finite dimensional then $End^{L(\Gamma)}(Q)$ is also finite dimensional for each summand $Q$ of $P$ since $End(Q)$ is isomorphic to a summand of $End(P)$. Using this fact we restrict the possibilities for these  summands. \\
 
As in the proof of $(iv) \Rightarrow (ii)$ of Proposition \ref{basit},  $End^{L(\Gamma)}(v_ZL(\Gamma)) \cong v_ZL(\Gamma))v_Z$ is infinite dimensional for all $v_Z$ since $v_ZL(\Gamma)v_Z$ contains the infinitely many nonzero (by Proposition \ref{nonzero}(i)) orthogonal idempotents $\{ ee^* \> \vert \> \> e \in s^{-1}(v) \setminus Z\}$ which are linearly independent. Hence $P$ has no summand of the form $v_ZL(\Gamma)$. \\

If $vL(\Gamma)$ is a summand of $P$ then $\Gamma_{v\> \leadsto}$ can not contain an infinite emitter $u$:  Otherwise  $End^{L(\Gamma)}(vL(\Gamma)) \cong vL(\Gamma)v$ would be infinite dimensional since  $vL(\Gamma)v$ contains $\{pee^*p^* \> \vert \>\> se=u \}$, a set of nonzero orthogonal idempotents where $p$ is a path with $sp=v$ and $tp=u$. \\

If there is a cycle $C$ in $\Gamma_{v\> \leadsto}$ then $End^{L(\Gamma)}(vL(\Gamma))$ is infinite dimensional: $sCL(\Gamma)sC \cong End^{L(\Gamma)}(sCL(\Gamma))$ contains the infinite linearly independent set $\{C^n\}_{n=1}^\infty \subset \mathbb{F}\Gamma \subset L(\Gamma)$ and $End^{L(\Gamma)}(sCL(\Gamma))$ is a summand of $End^{L(\Gamma)}(vL(\Gamma))$ by Fact \ref{iso}(i). Hence $\Gamma_{v\> \leadsto}$ has no cycle for any summand $vL(\Gamma)$ of $P$.\\

If $\Gamma_{v \> \leadsto}$ is acyclic and contains infinitely many branch vertices but neither cycles nor  infinite emitters then we can find paths $\{p_i\}_{i\in \mathbb{N}}$ in $\Gamma $ such that $sp_0=v$ and $tp_i=sp_{i+1}$ being distinct branch vertices for all $i\in \mathbb{N}$ by Lemma \ref{branch}. We can find an arrow $e_i$ with $se_i=tp_i$ but $e_i$ is not the initial arrow $p_{i+1}$ since $tp_i$ is a branch vertex.
Let $q_i=p_0p_1\cdots p_ie_i$. Now $\{q_iq_i^*\}_{i\in \mathbb{N}} $ is an infinite set of nonzero orthogonal idempotents in $vL(\Gamma)v \cong End^{L(\Gamma)}(vL(\Gamma))$, hence $End^{L(\Gamma)}(vL(\Gamma))$ is infinite dimensional.\\

For each $u$ in $\Gamma_{v\> \leadsto}$ we have the nonzero orthogonal (by Fact \ref{p*q} since $\Gamma_{v \>\leadsto }$ is acyclic) idempotents $\{pp^*  \> \vert \> p \in Path (\Gamma), \> \>  sp=v, \>\> tp=u \}$ in $vL(\Gamma)v\cong End^{L(\Gamma)}(vL(\Gamma))$. Hence there are only finitely many paths $p$ from $v$ to $u$.\\

Therefore $P$ is a finite direct sum where each summand is isomorphic to a $vL(\Gamma)$ with $\Gamma_{v\>\leadsto}$ acyclic, containing no infinite emitters and only finitely many branch vertices.  For each  such vertex $v$, we will show that $vL(\Gamma) \cong \bigoplus_{i=1}^n v_iL(\Gamma)$ where each $v_i$ is a sink or a leak. \\

If $v$ is a leak then we are done. If $v$ is not a leak and $\Gamma_{v\> \leadsto}$ has no branch vertex then there is a unique maximal path $p$ with $sp=v$ and $tp=w$ a sink. We get $v=pp^*$ by applying (CK2) repeatedly. Hence $wL(\Gamma) \stackrel{p \underline{\>\>}}{\longrightarrow} vL(\Gamma)$ is an isomorphism by Fact \ref{iso}. Hence $vL(\Gamma)\cong wL(\Gamma)$ and we are done. The only remaining case is that of $\Gamma_{v\>\leadsto}$ containing a  branch vertex.\\

Let $K$ be the maximum of the lengths of the paths $p$ with $sp=v$ and $tp$ a branch vertex. $vL(\Gamma)$ is isomorphic to $\bigoplus_{se=v} teL(\Gamma)$ by Fact \ref{iso}(iii). 
If $\Gamma_{te \>\leadsto }$ contains a branch vertex then we replace $teL(\Gamma)$ with $\bigoplus_{sf=te} tfL(\Gamma)$ and repeat. This process stops after $K+1$ steps when we obtain $vL(\Gamma) \cong \bigoplus_{i=1}^n v_iL(\Gamma)$ with each $v_i$ either a leak or there is a unique maximal path $p$ with $sp=v_i$ and $tp=w_i$ a sink. As above $v_iL(\Gamma) \cong w_iL(\Gamma)$ and we are done.\\

$(ii) \Rightarrow (iii):$ $End^{L(\Gamma)}(vL(\Gamma)) \cong vL(\Gamma)v$ by \cite[Lemma 4(iii)]{ko}. If $v$ is a sink or a leak then $vL(\Gamma)v\cong \mathbb{F}$ by Lemma \ref{2.2}(i). If $P\cong \bigoplus_{i=1}^n v_iL(\Gamma)$ with each $v_i$ a sink or a leak then  
$End^{L(\Gamma)}(\bigoplus_{i=1}^n v_iL(\Gamma)) \cong 
\bigoplus M_{k_j}(\mathbb{F})$ where the sum is over distinct sinks or inequivalent leaks $v_j$ and $k_j$ is the multiplicity of $v_jL(\Gamma)$ in $P$ (by \cite[Lemma 4(iv)]{ko} $vL(\Gamma)w=0$ if $v$ and $w$ are distinct sinks or inequivalent leaks). \\

$(iii) \Rightarrow (i):$ Since each matrix algebra is finite dimensional, a finite direct sum of matrix algebras is again finite dimensional. 
\end{proof}\\

 We are now ready to prove our main result giving a necessary and sufficient condition for the quotient $L(\Gamma)/J$ to be isomorphic to a Leavitt path algebra, in terms of the generators of $J$ related to the cycles in $\Gamma$. Let's recall from Theorem \ref{Ranga} above that there are 3 possibilities for the ideal $J \cap \mathbb{F}[C]$ of the subalgebra $\mathbb{F}[C] \cong \mathbb{F}[x]$ of $L(\Gamma)$: \\
 (1)  $J\cap \mathbb{F}[C] =\mathbb{F}[C]$ if and only if $sC \in J$, hence $C$ does not contribute to the set of canonical polynomial generators, the vertices on $C$ are in $J\cap V$.\\
(2) If $J\cap \mathbb{F}[C] =0$ then $C$ does not contribute to the set of canonical polynomial generators either. \\
(3) If $J\cap \mathbb{F}[C] $ is a nonzero proper ideal of $\mathbb{F}[C] \cong \mathbb{F}[x]$ then 
 $f_C(C)=C^n+\lambda_{n-1}C^{n-1} + \cdots + \lambda_0sC$ with $\lambda_0 \neq 0$ 
 is a generator of the ideal $J\cap \mathbb{F}[C]$ of $\mathbb{F}[C]$ and $f_C(C)$ is a 
 canonical polynomial generator of $J$. Such cycles have no exit in $\Gamma /J$.\\

We will call an ideal $J$ of $L(\Gamma)$ {\bf dlf} (for distinct linear factors) if each $f_C(x) \in \mathbb{F}[x]$ is a product of distinct linear factors for all canonical polynomial generators $f_C(C)$ of $J$. An ideal is graded if and only if it has no canonical polynomial generator, hence it is vacuously dlf. While being a dlf ideal seems to depend on the digraph $\Gamma$, the following theorem implies that it depends only on the algebraic structure of $L(\Gamma)$.   

\begin{theorem} \label{ana}
The quotient 
 $L(\Gamma)/J$ of a Leavitt path algebra 
 $L(\Gamma)$ is isomorphic to a Leavitt path algebra (over $\mathbb{F}$) if and only if $J$ is a dlf ideal of 
$L(\Gamma)$. 
\end{theorem}

\begin{proof} 
First we will prove the special case of the theorem when each $C$ with a canonical polynomial generator $f_C(C)$ of $J$ 
is a loop (with no exit) and the ideal 
$J$ contains no vertices (hence no elements of the form   $v^{J\cap V}$ either). Then we will reduce the general case to this. \\

When $L(\Gamma)/J \cong L(\Lambda)$ for some digraph $\Lambda$, consider the finitely generated projective $L(\Gamma)/J$-module $P:=(v_C+J)(L(\Gamma)/J)$ where $v_C=sC$. Hence $End(P) \cong (v_C+J)(L(\Gamma)/J)(v_C+J) \cong \mathbb{F}[x,x^{-1}] / (f_C(x))$  by \cite[Lemma 4(iii)]{ko}  and Lemma \ref{2.2}(ii). In particular $End(P)$ is finite dimensional. Hence $End(P)$ is isomorphic to a finite direct sum of matrix algebras over $\mathbb{F}$ by Theorem \ref{26}. Since $End(P)$ is commutative, all these matrix algebras are 1 dimensional, that is, $End(P)\cong \mathbb{F}^{n_C}$ as algebras where $n_C=deg(f_C)$. Each $f_C(x)$ is a product of distinct linear factors in $\mathbb{F}[x]$ by \cite[Fact 9]{ko}.\\

Conversely, assume that each 
 polynomial $f_C(x)$ has distinct roots $r_{Cj}\in \mathbb{F}$ for $j=1,2,\cdots, deg(f_C)$ where  $\{f_C(C)\}$ are the canonical polynomial generators of the ideal $J$. None of the $r_{C j}$ is $0$  because $f_C(0)\neq 0$. We construct the digraph $\Lambda$ by deleting  each loop $C$, replacing $v_C=sC$ with vertices $v_{Cj}$ where $j=1,2,\cdots , deg(f_C)$, 
 and replacing each arrow $e$ with $te=v_C$ by arrows $e_j$ such that $se_j=se$ and $te_j=v_{C j}$. The 
 homomorphism $\varphi : L(\Gamma) \longrightarrow L(\Lambda)$ maps each $v_C$ to $\sum_{j} v_{C j}$, each $e$ (respectively $e^*$) with $te=v_C$ to $\sum_j e_j$ (respectively $\sum_j e_j^*$), the arrow $C$ to 
 $\sum_j r_{Cj} \> v_{C j}$ and $C^*$ to $\sum_j \dfrac{v_{C j}}{r_{C j}}$. Finally $\varphi$ maps all the remaining vertices, arrows and dual arrows to themselves. It's routine to check that the relations are satisfied, so $\varphi$ is a homomorphism.
 \\

All vertices $v_{C k}$ in $\Lambda$ are in the image of $\varphi$ since $\varphi \big(\prod_{j\neq k} (C-r_{C j}v_{C } )\big)= \big(\prod_{j\neq k} (r_{C k}-r_{C j})\big)v_{Ck}$ and the roots $\{r_{Cj}\}$ of $f_C(x)$ are distinct. 
Moreover, $\varphi (e) \> v_{C k}=e_k$ and $v_{ C k}\> \varphi(e^*)=e_k^*$ for all $e$ in $\Gamma$ with $te=v_C$. Hence $\varphi$ is onto. Vertices in $\Gamma$ are sent to (sums of) vertices in $\Lambda$ and elements of the form $v_Z$ in $L(\Gamma)$ are sent to elements of this form in $L(\Lambda)$ under $\varphi$. Such elements are nonzero in $L(\Lambda)$ by Proposition \ref{nonzero}, hence $Ker \varphi$ does not contain elements of this form. \\

If $J \cap \mathbb{F}[C]=0 $ then $C$ remains a cycle in $\Lambda$, hence the restriction of $\varphi$ to the subalgebra $\mathbb{F}[C]$ is one-to-one (since $\mathbb{F}[C] \cong \mathbb{F}[x]$ with $C \leftrightarrow x$,  
in both $L(\Gamma)$ and $L(\Lambda)$). If $J \cap \mathbb{F}[C]\neq 0$ then $f_C(C)$ is a canonical polynomial generator (since $J$ contains no vertices, so $J \cap \mathbb{F}[C] \neq \mathbb{F}[C]$) and the restriction of $\varphi$ to $\mathbb{F}[C]$ is the composition 
$$\mathbb{F}[C] \cong \mathbb{F}[x] \longrightarrow \mathbb{F}^n \cong \bigoplus_{j} \mathbb{F}v_{C j} \hookrightarrow L(\Lambda)$$
where $x \mapsto (r_{Cj})$ under $\mathbb{F}[x] \longrightarrow \mathbb{F}^n$. Hence $Ker \varphi \cap \mathbb{F}[C]$ is generated by $f_C(C)$ for each such $C$. By Theorem \ref{Ranga}, $Ker \varphi =J$ and thus $L(\Gamma)/J \cong L(\Lambda)$. \\

To reduce the general case to the special case proven above, given an ideal $I$ of $L(\Gamma)$ first we divide $L (\Gamma)$ by the largest graded ideal $I'$ contained in $I$ which is generated by the vertices and elements of the form $v^H$ in $I$ where $H=I \cap V$. The quotient $L(\Gamma)/I'$ is isomorphic to a Leavitt path algebra $L(\Gamma')$ for some $\Gamma'$ by Theorem \ref{5.7}(ii) and  $I/I'$ corresponds to the ideal $J'$ of $L(\Gamma')$ which contains neither a vertex nor elements of the form $v^H$. The cycles $C$ corresponding the canonical polynomial generators and the canonical polynomial generators $f_C(C)$  themselves remain essentially the same for $\Gamma'$ and $J'$. Next we 
construct $\Gamma''$ using Lemma \ref{loop} by 
replacing the cycles $C$ in $\Gamma$ with the loops $e'$. Now $J'$ corresponds to the ideal $J$ generated by $\{ f_C(e') \}$. The ideal $J$ contains neither vertices nor elements of the form $v^H$. We are done when we replace $\Gamma$ with $\Gamma''$ and $I$ with $J$.
\end{proof}\\

If $J$ is a dlf ideal of $L(\Gamma)$ then $\Gamma \sslash J$ will denote the digraph $\Lambda$ constructed in the proof of Theorem \ref{ana} satisfying $L(\Gamma)/ J \cong L(\Lambda)$. If $\Gamma$ is row finite  and $J$ is an ideal of $L(\Gamma)$ then digraph $\Gamma / J$ is obtained from $\Gamma$ by deleting the vertices in J from $\Gamma$. If  $\Gamma$ has infinite emitters then $\Gamma /J$ involves breaking vertices as described just below Theorem \ref{5.7}. $\Gamma\sslash J$ is obtained from $\Gamma / J$ by replacing each cycle $C_i$ corresponding to a canonical generator $f_i(C_i)$ of $J$ with a loop $e_i$ as in Lemma \ref{loop} and then  replacing $e_i$ with $deg (f_i)$ number of sinks and each arrow ending at $se_i$  with $deg(f_i)$ arrows each ending at a distinct sink replacing $e_i$. Note that the digraph $\Gamma \sslash J$ depends only on $\{deg f_i\}$, not on the actual canonical polynomial generators $\{f_i\}$. So we can define the digraph $\Gamma \sslash J$ for any ideal $J$ as a function of the triple $(I, \beta, d)$, equivalently the quadruple $(H,S,\beta,d)$. However, the isomorphism 
between $L(\Gamma)/J$ and $L(\Gamma\sslash J)$ of Theorem \ref{ana}  exists only when $J$ is a dlf ideal. Also, there is a dlf ideal in $\mathcal{J}(I, \beta, d)$ if and only if $d(C) <  \vert \mathbb{F} \vert$ for all $C \in \beta$.\\

The space of the dlf ideals of a Leavitt path algebra along with the filtration, the corresponding stratification and parametrization of the strata given above 
is also a Morita invariant, just like the space of all ideals. Each dlf ideal corresponds uniquely to a triple $(X, \beta, \bar{\theta})$ where $X$ is a closed submonoid of $\mathcal{V}(L_{\mathbb{F}}(\Gamma))$ and $\beta$ is a subset of the isomorphism classes of nonsimple fgips in $\mathfrak{M}^X$ as before, while $\bar{\theta}$ is a function from $\beta$ to finite subsets of $\mathbb{F}^\times:= \mathbb{F}\setminus \{0\}$. The function $\theta$ of Remark \ref{Morita} is related to $\bar{\theta}$ as follows: 
$$\theta(C)= \prod_{\lambda \in \bar{\theta}(C)} (1-\lambda x) $$
for all $C \in \beta$.\\

Next we want to show that an arbitrary quotient of a Leavitt path algebra by its nilradical (or its Jacobson radical) we get a Leavitt path algebra if the field of the coefficients is large enough. This result was prompted by a question of  G. Corti$\tilde{n}$as during a talk by the first named author. 
\\

There are two obstructions to being a dlf ideal for an ideal $J$ of $L_{\mathbb{F}}(\Gamma)$.  The first is that the coefficient field $\mathbb{F}$ may not contain the roots of all $f_C(x)$ where $f_C(C)$ are the canonical polynomial generators of $J$. This may be overcome by extending the field of coefficients to an extension field $\mathbb{E}$ of $\mathbb{F}$ which contains all the roots of all $f_C(x)$, for instance the algebraic closure $\overline{\mathbb{F}}$ of $\mathbb{F}$, since $\mathbb{E} \otimes_{\mathbb{F}} L_{\mathbb{F}}(\Gamma) \cong L_{\mathbb{E}}(\Gamma)$. The second obstruction is that some $f_C(x)$ may have  repeated roots. This obstruction manifests itself as the nilradical of the quotient $L(\Gamma)/J$.\\ 

Let $g_C(x)$ be the product of the distinct irreducible factors of $f_C(x)$ with $g_C(0)=1$ and let  
$J'$ be the ideal generated by $J$ and all $g_C(x)$. We have a canonical epimorphism from $L(\Gamma)/J$ to $L(\Gamma)/ J'$ and the image of $g_C(C)+J$ is a nilpotent element of $L(\Gamma)/ J$.

\begin{lemma} \label{radical}
The kernel of the epimorphism from $L(\Gamma)/J$
  to $L(\Gamma)/J'$ consists of nilpotent elements.   
\end{lemma}

\begin{proof}
  Replacing $\Gamma$ with $\Gamma /J$ we may assume that the cycles $C$ contributing canonical polynomial generators have no exit. An arbitrary element $a \in L(\Gamma)$ is of the form $\sum \lambda_i p_iq_i^*$ with $p_i, \> q_i \in Path (\Gamma)$, $tp_i=tq_i$ and $\lambda_i \in \mathbb{F}$. If $q_i^*g_C(C)\neq 0$ then $sq_i=sC$ and $tq_i=tp_i$ is on $C$ since $C$ has no exit. If $p$ is the part of $C$ from $tp_i$ to $sC$ then $pp^* = tp_i$ by $(CK1)$ and $(CK2)$. Hence $p_iq_i^*=p_ipp^*q_i^*=p_ipC^{m_i} $ for some integer $m_i$. Here we use the convention that $C^0:= sC$ and $C^{-m}:= (C^*)^m$ when $m$ is a positive integer. If  $$P_C:= \{ p \in Path(\Gamma) \> \vert \> tp=sC \text{ and } p\neq qC \} $$
then $ag_C(C)= \sum \lambda_ip_i' C^{n_i} g_C(C)$ with $p_i' \in P_C$ and $n_i \in \mathbb{Z}$. Similarly for $a, \> b \in L(\Gamma)$ we have $ag_C(C)b= \sum \lambda_jp_j C^{n_j} g_C(C) q_j^*$ with $\lambda_j \in \mathbb{F}, \> \> p_j,\> q_j \in P_C$.\\

Having replaced $\Gamma$ with $\Gamma/J$ we now have $J'$ generated by $\{ g_C(C)\}$. Hence an element of $J'$ is of the form $\sum_{i} \sum_j \lambda_{ij}p_{ij} C_i^{n_{ij}} g_{C_i}(C_i) q_{ij}^*$ where $\{C_i\}$ is a finite subset of the cycles contributing canonical polynomial generators. If $i\neq k$ then $q_{ij}^*p_{kl}=0$ by Fact \ref{p*q}. Also $q_{ij}^*p_{il}\neq 0$ if and only if $q_{ij}=p_{il}$ by Fact \ref{p*q} again since $q_{ij}, \> p_{il} \in P_{C_i}$. If $q_{ij}=p_{il} $ then $$p_{ij}C_i^{n_{ij}}g_{C_i}(C_i)q_{ij}^* p_{il}C_i^{n_{il}}g_{C_i}(C_i)q_{il}^*= p_{ij}C_i^{n_{ij}+n_{ilk}}g_{C_i}(C_i)^2 q_{il}^*.$$
A similar formula holds for the product several such elements. In particular, if $g_{C_i}(C_i)^{r_i} \in J $ for some positive integer $r_i$ then a $r_i$-fold product of such elements is also in $J$. \\

Let $r_i$ be the smallest positive integer such that $g_{C_i}(C_i)^{r_i} \in J$ and let $M$ be the least common multiple of the $\{ r_i \}$. We see that each element in the kernel of the epimorphism from $L(\Gamma)/J$ to $L(\Gamma)/J'$ is nilpotent because  $$\big( \sum_{i} \sum_j \lambda_{ij}p_{ij} C_i^{n_{ij}} g_{C_i}(C_i) q_{ij}^* \big)^M \in J$$ by the discussion above. 
\end{proof}\\

\begin{theorem}\label{jake}
    If $\Gamma$ is an arbitrary digraph, $J$ an ideal of $L_{\mathbb{F}}(\Gamma)$, the field $\mathbb{F}$ contains all the roots of all canonical polynomial generators of $J$ and $J'$ is the ideal defined above
    then the Jacobson radical of $L_{\mathbb{F}}(\Gamma)/J$ equals the nilradical $N:=N(L_{\mathbb{F}}(\Gamma)/J)$ and $$(L_{\mathbb{F}}(\Gamma)/J)/ N \cong L_{\mathbb{F}}(\Gamma \sslash J').$$
    
    \end{theorem}

\begin{proof}
By Lemma \ref{radical} above every element of the kernel of the epimorphism from $L_{\mathbb{F}}(\Gamma)/J$ to $L_{\mathbb{F}}(\Gamma \sslash J')$ is nilpotent, hence the kernel is contained in the nilradical of $L_{\mathbb{F}}(\Gamma)/J$ and the nilradical (of any ring) is contained in the Jacobson radical.\\

Conversely, the image under an epimorphism of the Jacobson radical is the Jacobson radical and the Jacobson radical of a Leavitt path algebra is $0$ (see Remark \ref{nil}). Since $L_{\mathbb{F}}(\Gamma)/J' \cong L_{\mathbb{F}}(\Gamma \sslash J')$ by Theorem \ref{ana}, the Jacobson radical of $L_{\mathbb{F}}(\Gamma)/J$ is contained in the kernel and thus the kernel, the nilradical and the Jacobson radical are all equal to each other.    
\end{proof}\\

Since the Jacobson radical and hence the nilradical of a Leavitt path algebra is zero, a necessary condition for a quotient $L_{\mathbb{F}}(\Gamma)/J$ to be isomorphic to a Leavitt path algebra is that $L_{\mathbb{F}}(\Gamma)/J$ is reduced (that is, its nilradical is zero) or $L_{\mathbb{F}}(\Gamma)/J$ is semiprimitive (that is, its Jacobson radical is zero). This condition is also sufficient when the coefficient field is large enough by Theorem \ref{jake} .

\begin{corollary} \label{Null}
{{\bf (A Noncommutative Nullstellensatz)}} 
\label{primitif}
    When $\mathbb{F}= \overline{\mathbb{F}}$, the quotient $L_{\mathbb{F}}(\Gamma)/J$ is isomorphic to a Leavitt path algebra if and only if $L_{\mathbb{F}}(\Gamma)/J$ is reduced if and only if $L_{\mathbb{F}}(\Gamma)/J$ is semiprimitive.
\end{corollary}

\begin{remark} \label{32}
 The preceding corollary is the advertised noncommutative analog of the classical Nullstellensatz of commutative algebra/algebraic geometry stating that over an algebraically closed field a quotient of the algebra of regular functions on an affine variety is the algebra of regular functions on a subvariety if and only if its radical is zero. 
\end{remark}

\subsection{Genericity} \label{3.4}
When the coefficient field is algebraically closed $L(\Gamma)/J$ is isomorphic to a Leavitt path algebra if and only if the roots of every canonical polynomial generator of $J$ are distinct. This is a generic condition on the ideals of $L(\Gamma)$ in the sense explained in the theorem below.

\begin{theorem}
\label{generic}
    If $\mathbb{F}$ is algebraically closed, $I$ a graded ideal of $L(\Gamma)$ and $d$ a function from the subset $\beta$ of (geometric) cycles with no exit in $\Gamma/I$ to positive integers then the set of dlf ideals in each affine factor of  $\mathcal{J}(I,\beta, d)$ is open and dense with respect to the Zariski topology (and the Euclidean  topology when $\mathbb{F}=\mathbb{C}$). The topology on $\mathcal{J}(I,\beta)$ is generated by the topologies on the subsets $\mathcal{J}(I,\beta, d)$, that is, a subset of $\mathcal{J}(I,\beta)$ is open if and only if its intersection with each $\mathcal{J}(I,\beta, d)$ is open for all $d>0$.    
    Hence the property of being a dlf ideal is\enquote{generic} and therefore  
a generic quotient of a Leavitt path algebra  is isomorphic to a Leavitt path algebra. \end{theorem}

\begin{proof} In Theorem \ref{Ideal} we identified $\mathcal{J}(I,\beta, d)$ with $\prod_{C\in \beta}  \big(\mathbb{F}^{d(C)} \setminus \{0\} \big)$ where $(a_1,a_2,\cdots ,a_{d}) \in  \mathbb{F}^{d}\setminus\{0\} $ corresponds to the polynomial $f(x) = 1+ a_1x + a_2x^2+\cdots  +a_dx^d$ and $\{ f_C(C) \>\vert \> C\in \beta \} $ is the set of canonical polynomial generators of $J \in \mathcal{J}(I,\beta, d)$.  When $\mathbb{F}$ is algebraically closed the subset of the factor $\mathbb{F}^{d}\setminus \{0\}$ corresponding to 
polynomials $f(x)$ of positive degree with $f(0)=1$ and having distinct roots is the complement of the hypersurface defined by the discriminant in degree $d$ (which is the resultant of a general polynomial $f$ of degree $d$ and its derivative $f'$). Since the complement of a hypersurface is open and dense in the Zariski topology (and the Euclidean  topology when $\mathbb{F}=\mathbb{C}$), this is a generic condition. 
By the definition of the topology on $\mathcal{J}(I,\beta)$ the subset of dlf ideals in it is also open and dense. By Theorem \ref{ana} if all canonical polynomial generators of the ideal $J$ have distinct linear factors then the quotient $L(\Gamma)/J$ is isomorphic to a Leavitt path algebra. Hence a  quotient of a Leavitt path algebra by a \enquote{generic} dlf ideal is isomorphic to a Leavitt path algebra. 
 \end{proof}\\

When $\Gamma$ is finite or if $\Gamma$ has a finite number of exclusive cycles then $\mathcal{J}(I,\beta, d)$ is itself an affine variety for every triple $(I,\beta, d)$. When $\beta$ is infinite and  $\mathbb{F}=\mathbb{C}$ we may use a norm like $\lVert  \> \> \rVert_p$ on $\mathcal{J}(I,\beta)$ 
with $1 \leq p\leq \infty$ for an infinite product with a similar norm on each factor corresponding to a cycle $C\in \beta$ and observe that there are arbitrarily small perturbations (with respect to the chosen norm) moving an arbitrary given element of $\mathcal{J}(I,\beta)$ into our generic subset of dlf ideals. Note that for every dlf ideal $J'$ in $\mathcal{J}(I,\beta)$ we have $L(\Gamma)/J' \cong L(\Gamma \sslash J)$. Hence $L(\Gamma)/J$ is the limit of algebras isomorphic to $L(\Gamma \sslash J)$, that are deformations of $L(\Gamma)/J$, for an arbitrary ideal $J$ of $L(\Gamma)$. \\

A curious consequence of Theorem \ref{ana} is the following: 
\begin{corollary}
   \label{epi} Let $\varphi$ be an epimorphism from $L(\Gamma)$ to $L(\Gamma')$.\\
   (i) If $\Gamma'$ has neither a sink nor a leak then $Ker\varphi$ is a graded ideal of $L(\Gamma)$.\\
   (ii) If $\Gamma$ is row-finite, $\Gamma'$ has neither a sink nor a leak and $\varphi(v)\neq 0$ for every vertex $v$ then 
   $\varphi$ is an isomorphism.
    \end{corollary}
    \begin{proof} (i) If $J:=Ker \varphi$ is not graded then $J$ has a canonical polynomial generator and $\Gamma \sslash J$ has a sink as in the proof of Theorem \ref{ana}. Hence $L(\Gamma \sslash J)$ has a simple projective module by Proposition \ref{basit}. Since $L(\Gamma \sslash J) \cong L(\Gamma')$, we get that $L(\Gamma')$  has a simple projective module and $\Gamma'$ has a sink or a leak by Proposition \ref{basit} again. Therefore if $\Gamma'$ has neither a sink nor a leak then $Ker\varphi$ is a graded ideal. \\

    (ii) By part (i), $Ker\varphi$ is a graded ideal. Since $\Gamma$ is row-finite, graded ideals of $L(\Gamma)$ are generated by vertices by Theorem \ref{5.7}(i). Hence $Ker\varphi = 0 $.
    \end{proof}\\
    
  In Corollary \ref{epi} the hypothesis that $\Gamma'$ has no leak is necessary. When $\Gamma'$ has infinitely many vertices,  sinks and leaks are not distinguishable algebraically as illustrated in the example below. 
   
\begin{example} Consider the digraphs $\Gamma$ and $\Lambda$ shown below: 
$$\Lambda : \quad \xymatrix{   
 \cdots  \> \> \bullet_2 \ar@{->}[r]& \bullet_1 \ar@{->}[r] &{\bullet_0}} $$  

 $$\Delta :\quad \xymatrix{ {\bullet_0}
 \ar@{->}[r]   &{\bullet_1}
 \ar@{->}[r]& \bullet_2 \>\>  
 \cdots   } $$

$$\qquad \Gamma : \quad \xymatrix{   
 \cdots  \> \> \bullet_2 \ar@{->}[r]& \bullet_1 \ar@{->}[r] &{\bullet_v}\ar@(dr,ur)_e} $$\\ 
 
All digraphs have vertices indexed by $\mathbb{N}$. The digraph $\Lambda$ has a unique sink but no leak. All vertices of $\Delta$ are (equivalent) leaks and $\Delta$ has no sink. Since there is a unique arrow emitted from each nonsink vertex, (CK1) and (CK2) yield that $L(\Lambda)$ and $L(\Delta)$ are spanned by their paths and dual paths. We have an epimorphism from $L(\Lambda)$ to $M_\mathbb{N}(\mathbb{F})$, the algebra of matrices with rows and columns indexed by $\mathbb{N}$ having finitely many nonzero entries, mapping the path $p_{ij}$ from $i$ to $j$ to the elementary matrix $E_{ij}$ if $i\geq j$, and mapping the dual path $p_{ji}^*$ to $E_{ij}$ if $i<j$. Since the image of this spanning set is linearly independent, 
$L(\Lambda) \cong  M_\mathbb{N}(\mathbb{F})$. Similarly we have an isomorphism from $L(\Delta)$ to 
 $M_\mathbb{N}(\mathbb{F})$ mapping the unique path $p_{ij}$ from $i$ to $j$ with $i\leq j$ to $E_{ij}$ and $p_{ji}^*$ to $E_{ij}$ if $i>j$. Therefore 
 $L(\Lambda) \cong L(\Delta)$. As in Theorem \ref{ana} $L(\Gamma)/ (v-e) \cong  L(\Lambda) \cong L(\Delta)$. So we have an epimorphism from $L(\Gamma)$ to  $L(\Delta)$ and $\Delta$ has no sink, but kernel of this epimorphism is not graded. (If the kernel, that is, $(v-e)$ were graded then $v$ would also be in the kernel, but the ideal $(v-e)$ contains no vertices.) 
 
 \end{example}

\section{Quantum Spaces} \label{4}

The analytic version of the Hilbert Nullstellensatz is Gelfand-Naimark duality, stating that each commutative unital $C^*$-algebra is isometrically isomorphic to the $C^*$-algebra $C(X)$ of complex-valued continuous functions (with the sup norm) on a unique compact Hausdorff topological space $X$. In fact this correspondence can be extended to a cofunctor giving an equivalence between the category of compact Hausdorff topological spaces and the opposite category of commutative unital $C^*$-algebras. Both in the Nullstellensatz and Gelfand-Naimark duality 
the points of the variety/space correspond to the maximal ideals of the algebra and the closed sets are zero sets of ideals, equivalently, the set of maximal ideals containing a given ideal (in the algebraic case closed sets are Zariski closed, also called algebraic sets or affine varieties). The Stone-Weierstrass theorem provides a bridge between these two fundamental results, implying that the algebra of polynomial functions with complex coefficients on a compact real variety (such as a sphere) or a compact semi-algebraic set (such as a disk) is dense in the $C^*$-algebra of continuous functions. \\

Arbitrary $C^*$-algebras are (isometrically isomorphic to) a closed $*$-subalgebra of bounded linear operators $B(H)$ on a Hilbert space $H$ by the Gelfand-Naimark-Segal theorem. However,  neither the Hilbert space $H$ nor the subalgebra of $B(H)$ is uniquely determined, unlike the commutative case. An important subclass of $C^*$-algebras is graph $C^*$-algebras (analytic versions of Leavitt path algebras \cite{ra05}). There is a noncommutative Stone-Weierstrass theorem relating Leavitt path algebras and graph $C^*$-algebras: A Leavitt path algebra with complex coefficients is dense in the graph $C^*$-algebra of the same digraph  \cite[Corollary 7.6]{t07}. (This result is perhaps closer to a noncommutative analog of the original Weierstrass theorem stating that polynomial functions are dense in the $C^*$-algebra of complex valued continuous functions on the unit interval $[0,1]$, rather than the Stone-Weierstrass theorem which characterizes the dense $*$-subalgebras of the $C^*$-algebra of complex valued continuous functions on a compact Hausdorff space as unital $*$-subalgebras separating points.)\\

A \textit{quantum space} is a noncommutative $C^*$-algebra which is a deformation of the commutative $C^*$-algebra of complex valued continuous  functions $C(X)$ on a topological space $X$. Many important topological spaces (such as spheres $S^n$ and disks $D^{n}$) have quantum versions and several of these quantum spaces have been realized as graph $C^*$-algebras \cite{hs02}. The corresponding digraphs, which we will denote by $S^n_q$ and $D_q^{2n}$, are given below. The examples of the next section illustrating our theorems above will involve such quantum spheres and quantum disks.\\

$$D_q^{2n}:\quad    \xymatrix{ {\bullet} \ar@(ur,ul)_{C_1} \ar[r]& {\bullet} \ar@(ur,ul)_{C_2} \ar[r]& \> \> \cdots \> \> \ar[r]&{\bullet}\ar@(ur,ul)_{C_n}   \ar@(ur,ul) \ar[r] & {\bullet}
} $$

$$S_q^{2n}:\quad  \xymatrix{ {\bullet} \ar@(ur,ul)_{C_1} \ar[r]& {\bullet} \ar@(ur,ul)_{C_2} \ar[r]& \> \> \cdots \> \> \ar[r]&{\bullet}\ar@(ur,ul)_{C_n}   \ar@(ur,ul) \ar[r] \ar[dr]& {\bullet} \\
&&&& \> \> {\bullet}} $$

$$S_q^{2n-1}:\quad  \xymatrix{ {\bullet} \ar@(ur,ul)_{C_1} \ar[r]& {\bullet} \ar@(ur,ul)_{C_2} \ar[r]& \> \> \cdots \> \> \ar[r]&{\bullet}\ar@(ur,ul)_{C_n}   \ar@(ur,ul) } $$

\vspace{0.3cm}

\noindent
The completion of the Leavitt path algebra $L_{\mathbb{C}}(S_q^{2n-1})$ is isomorphic to the Vaksman–Soibelman (odd dimensional) quantum sphere. Similarly,  the completions of $L_{\mathbb{C}}(S_q^{2n})$ and $L_{\mathbb{C}}(D_q^{2n})$ are the (even dimensional) quantum sphere and the quantum disk. They were realized as graph $C^*$-algebras by Hong and Szymanski \cite{hs02}. The quantum 2-disk $C^*(D_q^{2})$
is also isomorphic to the celebrated Toeplitz algebra while $L_{\mathbb{F}}  (D_q^2)$ is isomorphic to the Jacobson algebra $\mathbb{F} \langle x,y\rangle / (1-xy)$ for every field $\mathbb{F}$. \\

The construction of a Leavitt path algebra from a digraph can be extended to a cofunctor from digraphs to algebras \cite{ht24}. The definition of a Leavitt path algebra gives the realization of this cofunctor $L_{\mathbb{F}}$ at the level of objects. The source category has digraphs as its objects and \textit{admissible} digraph morphisms as its morphisms. A  digraph morphism $\varphi: \Gamma \longrightarrow \Lambda$ is a function from the union $V_{\Gamma} \sqcup E_{\Gamma}$ of the vertices and arrows of $\Gamma$ to the union $V_{\Lambda} \sqcup E_{\Lambda}$ of the vertices and arrows of $\Lambda$ such that

(i) $\varphi (V_\Gamma) \subseteq V_\Lambda$ and $\varphi (E_\Gamma) \subseteq E_\Lambda$;

(ii) $\varphi(se)=
s\varphi(e) $ and $\varphi(te)=
t\varphi(e) $ for all arrows $e$ in $\Gamma$.\\
A digraph morphism $\varphi$
is \textit{admissible} if  

(i) $\varphi$ is finite-to-one;

(ii) $\varphi^{-1}(e) \stackrel{t}{\longrightarrow}\varphi^{-1}(te)$ is a bijection for all arrows $e$ in $\Lambda$;

(iii) $\varphi$ maps each sink to a sink or an infinite emitter.\\
A finite-to-one digraph morphism has to map an infinite emitter to an infinite emitter. Composition of admissible morphisms is admissible, hence digraphs and admissible digraph morphisms form a category. \\

Since $\varphi$ is finite-to-one we can define $L_{\mathbb{F}}(\varphi) :V_{\Lambda} \sqcup E_{\Lambda} \longrightarrow L_{\mathbb{F}} (\Gamma)$ by mapping each vertex and each arrow to the sum of their preimages. Conditions (ii) and (iii) of admissibility ensure that $L_{\mathbb{F}}(\varphi)$ extends to a graded $*$-algebra homomorphism from $L_{\mathbb{F}}
 (\Lambda)$ to $L_{\mathbb{F}}(\Gamma)$. When $char(\mathbb{F})\neq 2$ or when $\Gamma$ is row-finite, the conditions (ii) and (iii) of admissibility are also necessary for $L_{\mathbb{F}}(\varphi)$ to define an algebra homomorphism. \\

If $I$ is a graded ideal of $L_{\mathbb{F}}  (\Gamma)$ then there is an admissible  digraph morphism $\varphi: \Gamma /I \longrightarrow  \Gamma$ so that $L_{\mathbb{F}}(\varphi):L_{\mathbb{F}} (\Gamma) \longrightarrow L_{\mathbb{F}} (\Gamma/I) \cong L_{\mathbb{F}}   (\Gamma)/I$ is the quotient epimorphism: $\varphi$ is identity on the vertices in $V \setminus H$ as well as the arrows in 
 $\{ e \in E \> \vert \> te \notin H\} $ while 
$\varphi(v')=v$ and $\varphi(e')=e$ in the notation of Theorem \ref{5.7} with $H=I\cap V$ and $I=(H,S^H)$.\\

If $J$ is a non-graded dlf ideal of $L_{\mathbb{F}}  (\Gamma)$ then the epimorphism  $L_{\mathbb{F}}  (\Gamma) \longrightarrow L_{\mathbb{F}}  (\Gamma)/J \cong L_{\mathbb{F}}  (\Gamma \sslash J)$ of Theorem \ref{ana} is not a graded algebra homomorphism, therefore it can not come from an admissible digraph morphism. Also, if the dlf ideal $J$ is not graded then the digraph $\Gamma \sslash J$ has a sink but $\Gamma$ may have no sink or an infinite emitter (as in Examples \ref{Complex}, \ref{Ek}, \ref{ch=2}, \ref{21} of the next section), so there is no  admissible digraph morphism from $\Gamma \sslash J$ to $\Gamma$.

\begin{remark} Covariant functoriality of Leavitt path algebras has also been considered \cite{ht24-2}, but it seems less natural with more involved admissibility conditions. This is not so surprising since many fundamental functors (such as Gelfand-Naimark duality or its algebraic version, the prime spectrum of a ring) connecting geometry and algebra are contravariant. Even cofunctors with covariant counterparts like cohomology or K-theory appear more natural and have more structure (for example, products). Of course even more information can be extracted from the interaction of a contravariant functor with its covariant counterpart. 
    \end{remark}

If $\Gamma$ consists of a single loop $C$ and $f(x)$ is an arbitrary nonconstant polynomial with $f(0)=1$ which has distinct roots, and $J$ is the ideal generated by $f(C)$ then the geometric interpretation of $L_{\mathbb{F}} (\Gamma) \longrightarrow L_{\mathbb{F}} (\Gamma \sslash J)$ is the inclusion of finitely many points into a circle. If $\mathbb{F} = \mathbb{C}$ and  the roots of $f(x)$ are unit complex numbers then $L_{\mathbb{C}}(\Gamma)\cong \mathbb{C}[x,x^{-1}]$ may be identified with trigonometric polynomials and the $*$-homomorphism $L_{\mathbb{C}}(\Gamma) \longrightarrow L_{\mathbb{C}}(\Gamma \sslash J)$ with evaluating a trigonometric polynomial on the roots of $f(x)$. Completing $L_{\mathbb{C}}(\Gamma)$ with respect to the sup norm we get the $C^*$-algebra of complex-valued continuous functions on the circle and the restriction (or the evaluation) map to functions on the finite subset of the points.  \\

More generally when $J$ is a dlf ideal the contravariant geometric interpretation $\Gamma \sslash J \hookrightarrow \Gamma$ of $L_{\mathbb{F}}  (\Gamma) \longrightarrow L_{\mathbb{F}}  (\Gamma)/J \cong L_{\mathbb{F}}  (\Gamma \sslash J)$ seems to be an algebraic quantum cofibration. For example, the epimorphism $L_{\mathbb{F}}   (S_q^3) \longrightarrow L_{\mathbb{F}}  (S_q^2)$ in Example \ref{ch=2} corresponds to the restriction of a function on $S^3$ to the subspace $S^2$ associated with the inclusion of the sphere $S^2$ into the sphere $S^3$ of the classical case. In general the quantum analog of the inclusion of $S^{2n}$ as the equator of $S^{2n+1}$ can not be induced by an admissible digraph morphism since $S_q^{2n+1}$ does not have any sinks or infinite emitters, but it is realized as $L_{\mathbb{F}}  (S_q^{2n+1}) \longrightarrow L_{\mathbb{F}} (S_q^{2n+1} \sslash J)$ for an appropriate dlf ideal $J$ as in the second part of example 41 below.\\

We think a good category of algebraic quantum spaces whose objects are Leavitt path algebras should include the morphisms $L_{\mathbb{F}}  (\Gamma) \longrightarrow L_{\mathbb{F}} (\Gamma \sslash J)$ with $J$ a dlf ideal, in addition to morphisms coming from admissible digraph morphisms via the cofunctor $L_{\mathbb{F}}$. This would entail allowing a sink to map to a vertex on an exclusive cycle $C$ with some restrictions (such as: no arrow with source on $C$ is in the image of the digraph morphism).\\

The cofunctor $L_{\mathbb{F}}$ at the level of objects seems to be an algebraic quantum analog of Gelfand-Naimark theorem between compact Hausdorff  topological spaces and commutative unital $C^*$-algebras. In fact, the cofunctor $L_{\mathbb{C}}$ agrees with the classical Gelfand-Naimark cofunctor (assigning the $C^*$-algebra of complex valued continuous functions to a compact Hausdorff topological space) when restricted to the subcategory of finite discrete topological spaces (which are also finite digraphs with no arrows).\\

The appropriate analytic quantum analog would be $C^*(\Gamma)$, the graph $C^*$-algebra of $\Gamma$ \cite{ra05}. The Leavitt path algebra $L_{\mathbb{C}}(\Gamma)$ is a dense subset of $C^*(\Gamma)$ as mentioned above. The Leavitt path algebra $L_{\mathbb{F}} (\Gamma)$ corresponds to the algebra of polynomial functions in the classical case while  the morphism induced by the inclusion of a subvariety is the restriction of a function to the subvariety.\\

We also have algebraic quantum analogs of some semi-algebraic sets such as $D_q^{2n}$. For instance, in the topological category we have: 
$$S^1 *     \Omega_n \hookrightarrow  S^1*S^1= S^3 $$

\noindent
where  $\Omega_n \subset S^1$ is the set of $n$th roots of unity in the unit circle $S^1 \subset \mathbb{C}$ and $*$ denotes topological join. Hence $S^1 * \Omega_n$ is $n$ copies of the 2-dimensional disk $D^2$ with their boundary circles identified. Applying the Gelfand-Naimark cofunctor we get the commutative $C^*$-algebras of complex valued continuous functions where the inclusion becomes restriction:   
$$C(S^1 *     \Omega_n )\longleftarrow  C(S^1*S^1)= C(S^3) $$

\noindent
Polynomials in $\{x,y,\bar{x} , \bar{y} \}$ with complex coefficients are dense in $C(S^3)$
by the Stone-Weierstrass Theorem. Here $S^3= \{(x,y) \in \mathbb{C}^2 \> \vert \> x\bar{x} +y\bar{y}=1\}$ is identified with $S^1 *S^1$ via $[z,w,t]\leftrightarrow (\sqrt{1-t}\> z , \sqrt{t} \> w) \in S^3$ where $z,w \in S^1$ and $t\in [0,1]$. The algebraic quantum analog is 
$$ L_{\mathbb{F}}  (S_q^3)/(v-e^n) \cong L_{\mathbb{F}}  (S_q^3 \sslash (v-e^n)) \longleftarrow L_{\mathbb{F}}  ( S_q^3 ) $$ 
in the notation of Example \ref{n}. Completing with respect to an appropriate norm we get 
the quantum analog \cite[Corollary 7.6]{t07}
$$ C^*(S_q^3 \sslash (v-e^n)) \longleftarrow C^*( S_q^3 ) $$ 

\noindent
suggesting that $(S^1 *\Omega_n)_q= S_q^3 \sslash (v-e^n)$ of Example \ref{n} in the next section.\\

Of course an important aspect of Gelfand-Naimark duality is that it goes both ways: given a unital commutative $C^*$-algebra $A$ we  can find a compact Hausdorff topological space $X$ with $A\cong C(X)$ as $C^*$-algebras. Unfortunately in the Gelfand-Naimark-Segal theorem stating  that any $C^*$-algebra is isomorphic to a closed $*$-subalgebra of bounded linear transformations on some Hilbert space, neither the Hilbert space nor the closed $*$-subalgebra is unique. However, graph $C^*$-algebras and quantum spaces are rather special, so one may hope for some sort of uniqueness. \\

It is impossible to recover the digraph $\Gamma$ from the Leavitt path algebra $L_{\mathbb{F}} (\Gamma)$, there are many distinct digraphs whose Leavitt path algebras are isomorphic as in Lemma \ref{loop}. However, we may still hope to find a unique canonical representative digraph $\Gamma$ such that $L_{\mathbb{F}} (\Gamma)$ is isomorphic or Morita equivalent to the given Leavitt path algebra. (This is tantamount to classifying Leavitt path algebras up to isomorphism or up to Morita equivalence.) So far this has been accomplished only for Leavitt path algebras with Gelfand-Kirillov dimension $<4$  in \cite{ko22}.

\section{Examples}

Most algebraic properties of Leavitt path algebras turn out to be independent of the coefficient field (such as the Leavitt path algebra being unital or finite dimensional or Noetherian or Artinian or simple  or semisimple \cite{aam17} or having polynomial  or exponential growth \cite[Theorem 5]{aajz12} or having 
nonzero finite dimensional quotient \cite[Theorem 6.5]{ko1} or having the Invariant Basis Number property \cite{km} etc.). However, the quotient of $L(\Gamma)/J$ being isomorphic to a Leavitt path algebra does depend on the field of coefficients as shown in  Examples \ref{Complex} and \ref{ch=2} below. The remaining examples are to illustrate Theorem \ref{ana} and to indicate an intriguing connection with quantum spaces. Below we will not use the abbreviation $L(\Gamma)$ for  $L_{\mathbb{F}}(\Gamma)$ since we will need to consider different coefficient fields.

\begin{example} \label{Complex}
Consider the digraph
 $$\Gamma :\>  \xymatrix{ {\bullet}^{v} \ar@(dr,ur)_e &  } $$ \\
 Then $L_{\mathbb{F}}(\Gamma) \cong \mathbb{F}[x,x^{-1}]$.\\
 
 If $\mathbb{F}=\mathbb{R}$ then $L_{\mathbb{R}}(\Gamma)/(v+e^2)$ is not isomorphic to a Leavitt path algebra since the roots of the canonical polynomial $f(x)=1+x^2$ are not real. However, if we take $\mathbb{F}=\mathbb{C}$ then  $$L_{\mathbb{C}}(\Gamma)/(v+e^2) \cong L_{\mathbb{C}}(\Lambda) \cong \mathbb{C}^2.$$ 
 
 \noindent 
 where $\Lambda =\Gamma \sslash (v+e^2)$ is just two isolated vertices $v_1, \> v_2$.
A related example is the quotient map from $L_{\mathbb{F}}(\Gamma)$ to $L_{\mathbb{F}}(\Lambda)$ where $\Gamma$ is as above but $\Lambda$ is $n$ isolated vertices and the canonical polynomial, up to a scalar factor, is $f(x)=  (x-a_1) (x-a_2) \dots  (x-a_n)$ with $a_i \in \mathbb{F}$ distinct.
    \end{example}

\begin{example} \label{Ek}
$$ \Gamma:  \> \> \xymatrix{ {\bullet}  \ar[d] & {\bullet} \ar[l]\\
{\bullet_v} \ar[r]^e & {\bullet} \ar[u] } \qquad \quad 
 \xymatrix{   \Lambda: &
 \bullet \ar@{->}[r]& \bullet \ar@{->}[r]& \bullet \ar@{->}[r] &{\bullet_v}\ar@(dr,ur)_e\\
 &&&&\bullet\\
 \Gamma \sslash J  : &  
 \bullet \ar@{->}[r]& \bullet \ar@{->}[r] &\ar[ur]{\bullet} \ar[dr]\ar@{->}[r] & \bullet \\
  &&&&\bullet}
 $$

\noindent
In this example $f(x)=(x-\alpha)(x-\beta )(x-\gamma)$ and $J$ is the ideal generated by $(C-\alpha v)(C-\beta v)(C-\gamma v)$ where $\alpha, \beta, \gamma \in \mathbb{F}$ are distinct. We need to go through the intermediate step of turning the 4-cycle $C=eq$ into a loop with a tail, that is, replacing $\Gamma$ with $\Lambda$ using Lemma \ref{loop}. Here $L_{\mathbb{F}}(\Gamma) \cong L_{\mathbb{F}}(\Lambda) \cong M_4(\mathbb{F}[x,x^{-1}])$ and 
$L_{\mathbb{F}}(\Gamma \sslash J) \cong L_{\mathbb{F}} (\Gamma) /J \cong 
M_4(\mathbb{F}) \oplus M_4(\mathbb{F}) \oplus M_4(\mathbb{F})$.

\end{example}

  \begin{example} \label{19}
  \vspace{0.5cm}
$$S_q^2:\quad \xymatrix{ {\bullet}\ar@(ur,ul) \ar[r] \ar[dr]& {\bullet}_{w_1} \\
& \> \> {\bullet}_{w_2}} $$\\

Let $H:=\{w_2\}$ and $J=(w_2)$. $J$ has no canonical polynomial generators since $J$ is a graded ideal.
Then $L_{\mathbb{F}} (S_q^2)/J\cong L_{\mathbb{F}}   (S_q^2 /J)=L_{\mathbb{F}} (D_q^2)$, where \\

$$D_q^2:\quad \xymatrix{ {\bullet}\ar@(ur,ul) \ar[r] & {\bullet}_{w_1}
} $$
\end{example}

\begin{example} \label{ch=2}
\vspace{0.5cm}
$$S_q^5:\quad  \xymatrix{ {\bullet}_{v_1} \ar@(ur,ul)_{C_1} \ar[r]& {\bullet}_{v_2} \ar@(ur,ul)_{C_2} \ar[r]&{\bullet}_{v_3} \ar@(ur,ul)_{C_3}  } $$\\

\noindent
Let $H:=\{v_3\}$ and $f(x)=1-x^2$ and $J=(H, v_2-C_2^2)$.\\
So $f(x)=1-x^2=(1-x)(1+x)$.\\

\noindent
Deleting $H$  we get 

$$ S_q^5/ J \> = \> S_q^3 : \qquad   \xymatrix{ {\bullet}_{v_1} \ar@(ur,ul)_{C_1}\ar[r]& {\bullet}_{v_2} \ar@(ur,ul)_{C_2}}$$

\noindent
If $char(\mathbb{F}) \neq 2$ then 
$L_{\mathbb{F}}(S_q^5) /J\cong L_{\mathbb{F}} (S_q^5 \sslash J)=L_{\mathbb{F}} (S_q^2)$ where

$$S_q^5 \sslash J=
S_q^2 :\quad \xymatrix{ {\bullet}_{v_1} \ar@(ur,ul)_{C_1} \ar[r] \ar[dr]& {\bullet}_{v_{21}} \\
&{\bullet}_{v_{22}}}$$

If $Char(\mathbb{F})=2$ then $L_{\mathbb{F}} (S_q^5) /J$ is not isomorphic to a Leavitt path algebra because the roots of $f$ are not distinct.
\end{example}

\begin{example} \label{21}
Let the digraph $S_q^5$ and the hereditary saturated set $H:=\{v_3\}$ be as above,  but let 
$f(x)=1-x$, so $J=(H, v_2-C_2)$.

$$ S_q^5/J=S_q^3 : \qquad   \xymatrix{ {\bullet}_{v_1} \ar@(ur,ul)_{C_1} \ar[r]& {\bullet}_{v_2} \ar@(ur,ul)_{C_2}}$$\\

\noindent
$L_{\mathbb{F}}(S_q^5)/ J \cong L_{\mathbb{F}} (S_q^5\sslash J)= L_{\mathbb{F}}  (D_q^{2})$ where 
$$S_q^3 \sslash J = D_q^2 :\qquad \xymatrix{ {\bullet}_{v_1} \ar@(ur,ul)_{C_1} \ar[r]& {\bullet}_{v_2} }$$
\end{example}

\vspace{0.2cm}
\begin{example} \label{n} 
Let $\mathbb{F}$ be a field containing the $n$th roots of unity,  $char(\mathbb{F})$ not dividing $n$ and $f(x)=1-x^n$.

$$D_q^{4}:\quad    \xymatrix{ {\bullet} \ar@(ur,ul) \ar[r]& {\bullet}_v \ar@(ur,ul)_{e} \ar@(ur,ul) \ar[r] & {\bullet}_w
} $$

   \noindent
   If $J=(w, v-e^n)$ then 
$$D_q^{4}/J =S_q^3:\quad    \xymatrix{ {\bullet} \ar@(ur,ul) \ar[r]& {\bullet}_v \ar@(ur,ul)_{e} \ar@(ur,ul)} \qquad \qquad \qquad \qquad 
 $$

$$\textit{and } \qquad  D_q^4 \sslash J=S_q^3 \sslash (v-e^n) :\quad \qquad \xymatrix{ &\bullet_{v_1} \\
{\bullet} \ar@(ur,ul) \ar[ur] \ar[dr]& \> \> \vdots\\
&{\bullet}_{v_n}}  \qquad \qquad \qquad \qquad \qquad \qquad   \qquad  $$
   Hence $L_{\mathbb{F}}( D_q^4)/J \cong L_{\mathbb{F}} ( D_q^4 \sslash J) = L_{\mathbb{F}}(S_q^3 \sslash (v-e^n))$    
   by Theorem \ref{ana}. 
\end{example}

When $J$ is a graded ideal the map $L_{\mathbb{F}}  (\Gamma) \longrightarrow L_{\mathbb{F}}  (\Gamma)/J\cong L_{\mathbb{F}}  (\Gamma/J)$ is a graded $*$-algebra homomorphism. If the ideal $J$ is not graded then $L_{\mathbb{F}}  (\Gamma)/J \cong L_{\mathbb{F}} (\Gamma \sslash J)$ but this isomorphism can not be graded and usually it is not a $*$-algebra homomorphism when $\lambda^*= \lambda $, that is, when the involutive automorphism on $\mathbb{F}$ is identity. When $\lambda^*= \lambda $ it will be a $*$-algebra isomorphism only if all $f_C$ are $1-x$ or $1+x$ or $1-x^2$, when $char(\mathbb{F}) \neq 2$. In general $L_{\mathbb{F}}  (\Gamma)/J$ is $*$-algebra isomorphic to  $L_{\mathbb{F}}  (\Gamma \sslash J)$
if and only if the roots $\{ r \}$ of $f_C$ satisfy $r r^*=1$ for all canonical polynomial generators $f_C$.\\

 If $\mathbb{F}=\mathbb{C}$ then  $L_{\mathbb{C}}(\Gamma)$ is a dense subalgebra of the graph $C^*$-algebra $C^*(\Gamma)$ by \cite[Corollary 7.6]{t07} and the $*$-algebra structure that $L_{\mathbb{C}}(\Gamma)$ inherits from $C^*(\Gamma)$ is this $*$-algebra structure with 
$\lambda^*= \overline{\lambda}$ for all $\lambda \in \mathbb{C}$. In this case the condition on the roots of each $f_C$ is that they need to be distinct  nonzero complex numbers. \\

The digraphs in the examples above are all finite, hence their Leavitt path algebras are finitely generated. Moreover, the cycles in the digraphs of the examples  are pairwise disjoint. This geometric condition is equivalent to their Leavitt path algebras having finite Gelfand-Kirillov dimension, that is, having polynomial growth \cite[Theorem 5]{aajz12}. The dimensions of the quantum spaces in our examples equal the Gelfand-Kirillov dimensions of the corrresponding Leavitt path algebras \cite[Theorem 25]{ko5}. \\

\noindent
{\bf Acknowledgement:} 
This study was supported by Scientific and Technological Research Council of Türkiye (TÜBİTAK) under the Grants 124F214 and 122F414. The authors thank TÜBİTAK for their support.

\noindent
$^1$ Department of Mathematics \\ 
Gebze Technical  University,  Gebze, TÜRK\.{I}YE\\ 
E-mail: aytenkoc@gtu.edu.tr\\

\noindent
$^{2}$ Department of Mathematics \\ 
University of Oklahoma, Norman, OK, USA \\
E-mail: mozaydin@ou.edu

  \end{document}


\maketitle

\newcommand{\q}[1]{``#1''}

\medskip

\begin{abstract}
\noindent
Relevant facts that are (or should be) well known involving algebras and modules (not specifically  Leavitt path algebras) which are elementary but seem to be difficult to locate in the literature in the form that is needed, are given in this appendix which also contains a novel orthogonality relation between projective modules and ideals. 

\end{abstract}

{\bf Keywords:} Local units, projective modules, ideals, nonstable K-theory 
\medskip

\medskip

\section{Algebras with Local Units}
\medskip

A ring $A$ has \textit{local units} if every finite subset of $A$ is contained in a subring with $1$ (if $1=u \in A$ in this subring then $u=u^2$ and we say that $u$ is a local unit for this subset). Rings with local units are also  called idempotented rings \cite{ku}. Infinite direct sums of rings with 1, category algebras and their quotients (hence path algebras and Leavitt path algebras) have local units. Rings with local units mostly behave like rings with 1, see for instance \cite{1}. \textit{All rings below are assumed to have local units}.\\

 A \textit{unital} right $A$-module $M$ is a right $A$-module satisfying $M = MA$. (Unital left $A$-modules are defined analogously.) Since $A$ has local units $A=A^2$, so 
  $A$ is a unital left and a unital right $A$-module. If $M$ is a unital right $A$-module, $m = \sum m_i a_i $ for some (finitely many) $m_i \in M$, $a_i \in A$ and $u$ a local unit for $\{a_i\}$ then $mu=\sum m_i a_i u=\sum m_i a_i=m$. Hence there is a local unit $u \in A$ for all $m \in M$. If $\{m_i\} \subseteq M$
  and $\{a_j\} \subseteq A$ are finite subsets, $u_i$ is a local unit for $m_i$ and $v$ is a local unit for $\{a_j\}$ then a local unit $u$ for $\{u_i\} \cup \{v\}$ is a local unit for $\{m_i\} \cup \{a_j\}$. \textit{All modules below are assumed to be unital right modules unless specified otherwise}.

  \begin{fact} \label{no1}
      A ring $A$ is finitely generated as an $A$-module if and only if $A$ has $1$.
  \end{fact}
  \begin{proof}
      If $A$ has $1$ then $A$ is generated by $1$ as an $A$-module. When $A$ is generated by $\{ \alpha_i\}_{i=1}^n$ as a right $A$-module, let $u$ be a local unit for $\{ \alpha_i\}_{i=1}^n$. Hence $ua=a$ for all $a \in A$. For $a= \sum \alpha_i a_i \in A$ let $v$ be a local unit for $\{u\} \cup \{ a_i\}_{i=1}^n$. Now $u=uv=v$ and  $ua=a=av=au$ for all $a \in A$. Thus $u=1$.   
      \end{proof}\\

If $\bf{k}$ is a commutative ring with 1 then a $\bf{k}$-\textit{algebra} $A$ is both a ring and a left $\bf{k}$-module satisfying $\lambda(ab) = (\lambda a)b = a(\lambda b)$ for all $\lambda \in {\bf k}$ and for all $a, b \in A$. The classical case is when $\bf{k}$ is a field. A ring is also a $\mathbb{Z}$-algebra. Since $\bf{k}$ is commutative we get a $\bf{(k, k)}$-bimodule structure on A by defining $a \lambda := \lambda a$ for all $\lambda \in {\bf k}$ and for all $a \in A$. If $A$ is a $\bf{k}$-algebra and $a^2=a \in A$ then $aAa$ is a $\bf{k}$-subalgebra with $1=a$.\\

 If $A$ is a $\bf{k}$-algebra and $M$ is an $A$-module then $M$ is a $({\bf k}, \> A)$-bimodule via $\lambda m:= m (\lambda u)$ for all $m \in M$ with $u$ a local unit for $m$. This is well-defined: if $m=mu=mv$ for local units $u, v$ and $w$ is a local unit for $\{u, v\}$ then $m (\lambda u)= m(\lambda uw)=m(u \lambda w)=mu(\lambda w)= mv(\lambda w)=m(\lambda vw)=m(\lambda v)$. For all $\lambda \in {\bf k}$, $m\in M$, $a \in A$ and $u$ a local unit  for $m$ and $a$ (hence for $ma$ also) we have $\lambda (ma)=(ma)(\lambda u)=m(a\lambda u)=m(\lambda au)=m(\lambda a)=
m(\lambda ua)=(\lambda m) a$.\\

 If $A$ is a ${\bf k}$-algebra 
then an ideal $I$ of the ring $A$ inherits a ${\bf k}$-module structure from $A$ which is compatible with the ${\bf k}$-module structure defined above since $\lambda a = \lambda (au) = a (\lambda u)$ for all $\lambda \in {\bf k}$, for all $a \in A$ and $u$ a local unit for $a$. Ideals of $A$ are unital left and right $A$-modules since $IA=I=AI$. If $M$ and $N$ are $A$-modules then $Hom (M, \> N)$ is the $\bf{k}$-module of $A$-module homomorphisms where $(\lambda f)(m):=f(\lambda m)$ for all $\lambda \in \bf{k}$, $f \in Hom(M, N)$ and $m \in M$. Since $M$ is a $({\bf k}, \> A)$-bimodule $\lambda f$ is an $A$-module homomorphism. In particular $End(M):=Hom(M, M)$ is a $\bf{k}$-algebra. 
 
 \begin{fact} \label{rEnd}
     $Hom(M, N)$ is  a right $End(M)$-module via $(f \theta)(m):=f(\theta (m))$ for all $f \in Hom(M, N), \> \theta \in End(M)$ and $m \in M$. \end{fact}

\textit{All algebras below are $\bf{k}$-algebras with $\bf{k}$ fixed but not specified}.

\begin{fact}\label{summand}
    If $M$ is a summand of the $A$-module $N$ then $End(M)$ is isomorphic to a subalgebra of $End(N)$.
\end{fact}
\begin{proof}
Every endomorphism of $M$ can be extended by $0$ on a  complement of $M$ to get an endomorphism of $N$.
\end{proof}

\section{Projective Modules and Ideals}
The main goal of this section is to describe a Galois connection between the closed submonoids of $\mathcal{V}(A)$ defined below, and the lattice of ideals of $A$, defined via an orthogonality relation between projective $A$-modules and the ideals of $A$.  

\begin{lemma} \label{End}
If $M$ is an $A$-module for a $\bf{k}$-algebra $A$ and $a^2=a \in A$ then:   \\ 
(i) $aA$ is a finitely generated unital projective $A$-module.\\
(ii) $Hom(aA, \> M)\cong Ma=\{m \in M \> \vert \> ma=m\}$ as $aAa$-modules.\\
(iii) $End(aA) \cong aAa = \{ x\in A \> \vert \> axa=x\}$ as $\bf{k}$-algebras. \\
(iv) If $b^2=b \in A$ then $Hom(aA,\> bA) \cong bAa$ as $aAa$-modules.
\end{lemma}
\begin{proof}
(i) $A = A^2$ so $(aA)A=aA$, hence  $aA$ is a unital $A$-module and $aA$ is generated by $\{a\}$, hence it is finitely generated. If $f:aA \longrightarrow N$ and $g: M  \longrightarrow N$ are homomorphisms with $g$ onto then there is an $m\in M$ with $f(a)=g(m)$. Let $h=m \underline{\>\>}: aA \longrightarrow M$ be the homomorphism given by left multiplication with $m$. We get $f=g\circ h$ because $g(h(x))=g(mx)=g(m)x=f(a)x=f(ax)=f(x)$ for all $x \in aA$. Hence $aA$ is projective. \\

\noindent
(ii) The isomorphism from $Hom(aA,\>  M)$ to  $Ma=\{m \in M \> \vert \> ma=m\}$ is $f \mapsto f(a)=f(a^2)=f(a)a$  and its inverse is $m \mapsto m\underline{\>\>}$ since $f(x)=f(ax)=f(a)x$ for all $x \in aA$. We define  $fr$ for all $f\in Hom(aA,M)$ and $r \in aAa$ as $(fr)(x):=f(r x)$ for all   $x \in aA$ and realize $Hom(aA,\> M)$ as an $aAa$-module. The isomorphism given above is an $aAa$ module isomorphism since $(f r )(a)=f(r a)=f(ar ) = f(a)r$ because $ra = ar$ for all $r \in aAa$.\\

\noindent
(iii) This follows from (ii) and the fact that 
$(b\underline{\>\>}\>) \circ (c\underline{\>\>}\> ) = (bc) \underline{\>\>}$ with $b, c \in aAa$. \\

\noindent
(iv) This follows from (ii) with  $M=bA$.
\end{proof}\\

When $A$ has $1$, every finitely generated $A$-module is a quotient of a finitely generated free $A$-module. But non-zero free $A$-modules are never finitely generated if $A$ does not have $1$ by Fact \ref{no1}. However, finitely generated $A$-modules are still quotients of finitely generated projective $A$-modules.

\begin{fact}
    If the $A$-module $M$ is generated by $\{m_i\}_{i=1}^n$ then $M$ is a quotient of $(uA)^n$  
    where $u$ is a local unit for $\{m_i\}_{i=1}^n$. 
\end{fact}
\begin{proof}
    The epimorphism from $(uA)^n$ onto $M$ is $\bigoplus_{i=1}^n (m_i \> \underline{\>\>}\> )$
\end{proof}\\

We say that a projective $A$-module $P$ is {\bf orthogonal} to an ideal $I$ of $A$, denoted $P\perp I$, 
when $Hom(P, A/I)=0$. Clearly orthogonality depends only on the isomorphism type of $P$, so $[P] \perp I$ is well-defined where $[P]$ denotes the isomorphism class of $P$. There is a pre-order $\twoheadrightarrow$ on isomorphism classes of projective $A$-modules defined as 
$[P]\twoheadrightarrow [Q]$ if there is an epimorphism from $P$ to $Q$, equivalently if $Q$ is isomorphic to a summand of $P$. The set of ideals of $A$ orthogonal to $P$ is denoted by $P^\perp$ or $[P]^\perp$. Similarly, if $X$ is a collection of isomorphism classes of projective $A$-modules then $X^\perp$ denotes the set of ideals of $A$ orthogonal to all $[P] \in X$, that is, $X^\perp =\cap \{ P^\perp \> \vert \> [P] \in X \}$.   The set of isomorphism classes of  projective $A$-modules orthogonal to the ideal $I$ is denoted by $^\perp I$.

\begin{proposition}
     \label{perp}
Let $P$ and $Q$ be projective $A$-modules, $I$ and $J$  ideals of $A$, $X$ a collection of isomorphism classes of projective $A$-modules and $Y$ a collection of ideals of $A$.\\
  (i) If $P \twoheadrightarrow Q$ then $P^\perp \subseteq Q^\perp$. \\
  (ii) If $P\perp I$ and $I \subseteq J$ then $P \perp J$.\\
(iii) $Y \subseteq P^\perp$ if and only if $\> P \perp  \cap I $ where the intersection is over all $I \in Y$.\\
(iv) $X \subseteq  {^\perp} I$ if and only if  $\> \oplus P \perp I$ where the sum is over all $[P]\in X$.\\
(v) $aA \perp I$ if and only if $a \in I$ for every idempotent $a\in A$.\\
(vi) If $\{a_i \> | \> i \in \Lambda\}$ is a set of idempotents in $A$ then  $(\bigoplus_ {i \in \Lambda} a_iA)^\perp$ is the set of ideals in $A$ containing $\{a_i \> | \> i \in \Lambda \}$.
\end{proposition}
\begin{proof}
(i) Let $g: P \rightarrow Q$ be an epimorphism. If $I \in P^\perp$ and $f \in Hom(Q, \> A/I)$ then the composition $f \circ g: P \rightarrow Q \rightarrow A/I$ is $0$. Since $g$ is onto $f=0$, hence $I \in Q^\perp$.\\\\
\noindent
(ii) We have an epimorphism $f:A/I \rightarrow A/J$. If $h\in Hom(P, \> A/J)$ then there is a homomorphism $g: P \rightarrow A/I$ with $h=f\circ g$ since $P$ is projective. But $P \perp I$, hence $g=0$ and so $h=0$. Thus $P \perp J$.\\\\
\noindent
(iii) If $\> P \perp  \cap I$ then $P \perp I$ for all $I\in Y$ by (ii) above. Hence $Y\subseteq P^\perp$. Conversely, if $\>Y \subseteq P^\perp$ then $Hom(P, \>  A/I) =0$ for all $I\in Y$. We have a monomorphism $f: A/\cap I \rightarrow \prod A/I$ where the intersection and the product are both over $I\in Y$. If 
$g: P \rightarrow A/\cap I$ then the composition $f\circ g =0$ because $Hom(P, \> \prod A/I)\cong \prod Hom(P, \> I)=0$. Since $f$ is one-to-one, $g=0$. Hence $\> P \perp  \cap I$.\\

\noindent
(iv) Since $Hom(\oplus P , \> A/I) \cong \prod Hom(P, \> A/I)$ where the sum and the product are both over $P \in X$, the claim follows.\\

\noindent
(v) $Hom(aA,A/I) \cong (A/I)a$ by Lemma \ref{End}(ii) and $(A/I)a =0$ if and only if $a \in I$.\\

\noindent
(vi) $Hom(\bigoplus a_iA, \> A/I) \cong \prod Hom(a_iA, \> A/I) \cong \prod (A/I)a_i$ by Lemma \ref{End}(ii) 
where the sum and the product are over $i \in \Lambda$. Also, $(A/I)a_i =0 $ if and only if $a_i \in I$. Hence $Hom(\bigoplus a_iA, \> A/I)=0$ if and only if $\{a_i \> | \> i \in \Lambda \} \subseteq I$.   Thus  $(\bigoplus a_iA)^\perp $ is the set of ideals of $A$ containing $\{a_i \> | \> i \in \Lambda \}$.
\end{proof}\\

A \textit{Galois connection} is a pair $(\varphi, \psi)$ of order reversing maps $\varphi :\mathcal{P} \longrightarrow \mathcal{Q}$ and $\psi: \mathcal{Q} \longrightarrow \mathcal{P}$ between posets $\mathcal{P}$ and $\mathcal{Q}$ satisfying $\psi\varphi(x) \geq x$ for all $x\in \mathcal{P}$ and $\varphi \psi(y) \geq y$ for all $y \in \mathcal{Q}$. Since $\psi \varphi (x) \geq x$ and $\varphi$ is order reversing $\varphi \psi \varphi (x) \leq \varphi (x)$, but $\varphi \psi (\varphi(x)) \geq \varphi(x)$ for all $x \in \mathcal{P}$. Hence $\varphi \psi \varphi = \varphi$ and similarly $\psi \varphi \psi =\psi$. Therefore the restriction of $\varphi$ to $\psi (\mathcal{Q)}$ and the restriction of $\psi$ to $\varphi(\mathcal{P})$ are inverses of each other. These bijections give the \textit{Galois correspondence} between $\varphi(\mathcal{P})$ and $\psi(\mathcal{Q})$. If $\varphi (\mathcal{P})$ or $\psi(\mathcal{Q})$ is a lattice then so is the other and the Galois correspondence is a lattice anti-isomorphism (since $\varphi (\mathcal{P})$ and $\psi(\mathcal{Q})^{op}$, the opposite or the upside down poset of $\psi(\mathcal{Q})$ are poset isomorphic).  \\

The \textit{nonstable K-theory} of a ring $A$ is the additive monoid $\mathcal{V}(A)$ of isomorphism classes of finitely generated projective $A$-modules under direct sum. 
A submonoid $S$ of $\mathcal{V}(A)$ is \textit{closed} if $[P] \in S$ and $[P] \twoheadrightarrow [Q]$ implies that $[Q] \in S$. Arbitrary intersections of closed submonoids are closed submonoids. The closed submonoid generated by $X \subseteq \mathcal{V}(A)$, denoted by $\overline{X}$, is the intersection of all closed submonoids of $\mathcal{V}(A)$ 
containing $X$, that is, the smallest closed submonoid containing $X$.  

\begin{theorem}
    \label{closure}
There is a Galois connection between the subsets of $\mathcal{V}(A)$ under inclusion and the ideals of $A$ under reverse inclusion, given by $\psi(I):= {^\perp I}$ for all $I \trianglelefteq A$ and $\varphi(X):=\cap I$ for all $X\subseteq \mathcal{V}(A)$ where the intersection is over $ I\in X^\perp$. 
Moreover, $\psi(I)$ is a closed submonoid of $\mathcal{V}(A)$ for every $I \trianglelefteq A$.   
    When $P$ is a (not necessarily finitely generated) projective $A$-module let 
$$P_{\twoheadrightarrow} := \{ [Q] \in \mathcal{V}(A) \> \vert \> Q \textit{ is a quotient of  } P^n  \textit{ for some } n>0 \}.$$
If $X=\{ [Q_i] \} _{i \in \Lambda}$ and $P:= \oplus_{i\in \Lambda} Q_i$
then $P_{\twoheadrightarrow} =\overline{X}$ and ${X^\perp = \overline{X}^\perp =P^\perp}$. \end{theorem}
    \begin{proof}
Both $\varphi$ and $\psi$ are order reversing and $\varphi \psi (I) 
\subseteq I$ since $I \in (^\perp I)^\perp$. If $[P] \in X$ then 
 $P \perp \varphi (X)$ by Proposition \ref{perp}(iii) with $Y= X^\perp$.
Hence $P \in \psi \varphi (X)$, so $X \subseteq \psi \varphi (X)$. Thus $(\varphi, \psi)$ is a Galois connection. Also $\psi (I)$ is a closed submonoid of $\mathcal{V}(A)$ for all $I \trianglelefteq A$ by Proposition \ref{perp}(i) and (iv). \\

$P_{\twoheadrightarrow}$ is closed under addition and a quotient of a quotient of $P^n$ is a quotient of $P^n$, hence $P_{\twoheadrightarrow}$ is a closed submonoid of $\mathcal{V}(A)$. Clearly $[Q_i] \in P_{\twoheadrightarrow}$ for all $i \in \Lambda$, hence $\overline{X} \subseteq P_{\twoheadrightarrow}$. Any finitely generated quotient of $P^n$ is a quotient of $Q^n$ where $Q= \oplus_{i \in \Phi} Q_i$ for some finite $\Phi \subseteq \Lambda$.  Since $Q \in \overline{X}$ we get $P_{\twoheadrightarrow} \subseteq \overline{X}$, so  $P_{\twoheadrightarrow}= \overline{X}$.\\

Since $Hom(\oplus_{i \in \Lambda} Q_i, A/I) \cong \prod_{i \in \Lambda} Hom(Q_i, A/I)$, the ideal $I$ is in $P^\perp $ if and only if $I \in Q_i^\perp$ for all $i \in \Lambda$, that is,  $P^\perp = X^\perp$. For $n>0$ since $Hom(P^n, A/I) \cong Hom(P, A/I)^n$ we have $P^\perp =(P^n)^\perp$. If $I \in (P^n)^\perp$ for all $n>0$ then $I \in (P_{\twoheadrightarrow})^\perp $ by Proposition \ref{perp}(i). Hence $X^\perp = P^\perp \subseteq (P_{\twoheadrightarrow})^\perp =\overline{X}^\perp$. By definition $X \subseteq \overline{X}$ and so $\overline{X}^\perp \subseteq X^\perp$. Thus $X^\perp =\overline{X}^\perp=P^\perp$.
\end{proof}

\section{Maximal Ideals of Products and $\mathbb{F}[x,x^{-1}]$}

Finally, we have a few facts about maximal ideals of rings with $1$ and some elementary facts specifically about the Laurent polynomial algebra $\mathbb{F}[x,x^{-1}]$.

\begin{fact}\label{max}
    If $R$ and $S$ are rings with $1$ then maximal ideals of $R\times S$ are either of the form $I\times S$ or $R\times J$ where $I$ (respectively $J$) is a maximal ideal of $R$ (respectively $S$). Hence maximal ideals of $R_1\times R_2 \times \cdots \times R_n$
are of the form  $R_1\times \cdots R_{k-1} \times I\times R_{k+1} \cdots \times R_n$   with $1\leq k\leq n$ 
where each $R_i$ is a ring with $1$ and $I $ is a maximal ideal of $R_k$.
\end{fact}
\begin{proof}
We identify $R$ with $\{(r,0)\}_{r \in R} \subseteq R\times S$, similarly $S$ with $\{(0,s)\}_{s \in S} \subseteq R\times S$. If $\mathfrak{m}$ is an ideal of $R\times S$ then $R\times S\longrightarrow R/(R\cap \mathfrak{m}) \times S/(S\cap \mathfrak{m})$ is an epimorphism where $(r,s)\mapsto (r+ R\cap \mathfrak{m}, \> s+S\cap \mathfrak{m})$ and we identify $R\cap \mathfrak{m}$ and $S\cap \mathfrak{m}$ with ideals of $R$ and $S$, respectively. The kernel of this epimorphism is $\mathfrak{m}$ since $(r,s) \in \mathfrak{m}$ implies that $(r,0)=(r,s)(0,1) \in \mathfrak{m}$ and similarly that $(0,s) \in \mathfrak{m}$. Hence the induced homomorphism $(R\times S)/\mathfrak{m} \longrightarrow R/(R\cap \mathfrak{m}) \times S/(S\cap \mathfrak{m})$ is an isomorphism.\\

We have a one-to-one correspondence between the ideals of $R\times S$ containing $\mathfrak{m}$ and the ideals of $R/(R\cap \mathfrak{m}) \times S/(S\cap \mathfrak{m})$ via this isomorphism. If $\mathfrak{m}$ is a maximal ideal of $R\times S $ and $R\cap \mathfrak{m}\neq R$ then the projection $R/(R\cap \mathfrak{m}) \times S/(S\cap \mathfrak{m}) \longrightarrow R/(R\cap \mathfrak{m})$ must be an isomorphism (otherwise its kernel would be a proper non-zero ideal). Hence $S\cap \mathfrak{m}=S$ and $I:=R\cap \mathfrak{m}$ is a maximal ideal. Similarly, if $S\cap \mathfrak{m}\neq S$ we conclude that the maximal ideal $\mathfrak{m}$ is of the form $(R\cap \mathfrak{m} )\times S$. The second statement follows by mathematical induction.
  \end{proof}

\begin{fact} 
\label{distinct}
If $I$ is a proper nonzero ideal of $\mathbb{F}[x,x^{-1}]$ then 
$I$ is generated by a unique monic $f(x) \in \mathbb{F}[x]$ with $f(0) \neq 0$ and 
$\mathbb{F}[x,x^{-1}]/I \cong \mathbb{F}[x]/(f(x))$. If $f(x)= \prod_{i=1}^k f_i(x)^{m_i}$ where $ f_i(x)$ are distinct irreducible polynomials and $m_i$ are positive integers then $\mathbb{F}[x,x^{-1}]/I \cong \prod_{i=1}^k \mathbb{F}[x]/(f_i(x)^{m_i})$. In particular, $\mathbb{F}[x]/(f(x))\cong \mathbb{F}^m$ as $\mathbb{F}$-algebras (where the multiplication on $\mathbb{F}^m$ is coordinate-wise) if and only if $f(x)$ is a product of distinct linear factors. 
    \end{fact}
  
 \begin{proof}
  There is a $g(x) \in \mathbb{F}[x,x^{-1}]$ with $I=(g(x))$ 
   because $\mathbb{F}[x,x^{-1}]$ is a PID. Multiplying $g(x)\neq 0$ with a suitable power of $x$ (which is a unit in $\mathbb{F}[x,x^{-1}]$) and an element in $\mathbb{F}^\times:= \mathbb{F}\setminus \{0\}$, we obtain $I=(f(x))$ with $ f(x) \in \mathbb{F}[x]$ monic and $f(0) \neq 0$. The polynomial $f(x)$ is unique since the units in $\mathbb{F}[x,x^{-1}]$ are of the form $\lambda x^n$ with $\lambda \in \mathbb{F}^\times$ and $n \in \mathbb{Z}$.\\

The inclusion $\mathbb{F}[x] \hookrightarrow \mathbb{F}[x,x^{-1}]$ induces the monomorphism $\mathbb{F}[x]/(f(x)) \longrightarrow \mathbb{F}[x,x^{-1}]/(f(x))$. 
Let $f(x)=1-xh(x)$ with $h(x)\in \mathbb{F}[x]$. Then $x^{-1}+(f(x))= h(x)+(f(x)) \in \mathbb{F}[x,x^{-1}]/(f(x))$. Hence $\mathbb{F}[x,x^{-1}]/(f(x)) \cong \mathbb{F}[x]/(f(x))$. \\

If $f(x)= \prod_{i=1}^k f_i(x)^{m_i}$ where $ f_i(x) \in \mathbb{F}[x]$ are distinct irreducible polynomials and $m_i$ are positive integers then $\mathbb{F}[x,x^{-1}]/(f(x)) \cong \mathbb{F}[x]/(f(x))\cong \prod_{i=1}^k \mathbb{F}[x]/(f_i(x)^{m_i})$ by the Chinese Remainder Theorem.\\

If $(f(x))=(\prod_{i=1}^{k} (x-r_i))$ with distinct $r_i \in \mathbb{F}$ then $x \mapsto (r_1,r_2,\cdots r_k)$ defines an algebra isomorphism from $\mathbb{F}[x, x^{-1}]/(f(x))$ onto $\mathbb{F}^k$. For the converse, note that $\mathbb{F}^m$ has exactly $m$ maximal ideals (corresponding to a coordinate being 0). Since $\mathbb{F}[x]/(f_i(x)^{m_i})$ has a unique maximal ideal, namely $(f_i(x))+ (f_i(x)^{m_i})$, the algebra 
 $\mathbb{F}[x]/(\prod_{i=1}^k f_i(x)^{m_i})$ has $k$ maximal ideals by Fact \ref{max}. If $\mathbb{F}[x]/(\prod_{i=1}^k f_i(x)^{m_i}) \cong \mathbb{F}^m$ then 
 $deg \prod_{i=1}^k f_i(x)^{m_i}= dim^{\mathbb{F}} \>  \mathbb{F}[x]/(\prod_{i=1}^k f_i(x)^{m_i})= dim^{\mathbb{F}} \> \mathbb{F}^m = m=k$ by the count of maximal ideals in Fact \ref{max}. Hence $m_i=1$ and $deg f_i=1$ for $i=1,\cdots , k$. Since $f_i$ are distinct, we are done. 
\end{proof}\\

\noindent
{\bf Acknowledgement:} 
This study was supported by Scientific and Technological Research Council of Türkiye (TUBITAK) under the Grants 124F214 and 122F414. The author thanks TUBITAK for their support.

\noindent 
Ayten Ko\c{c}\\
Department of Mathematics\\
Gebze Technical  University, Gebze, TÜRK\.{I}YE\\ 
E-mail: aytenkoc@gtu.edu.tr\\